\pdfoutput = 1 
\documentclass[a4paper,11pt]{article}

\usepackage[utf8]{inputenc}
\usepackage[T1]{fontenc}
\usepackage{lmodern}
\usepackage{microtype}
\usepackage{amssymb}
\usepackage{amsfonts}
\usepackage{amsmath}
\usepackage{amsthm}
\usepackage{authblk}
\usepackage{comment}
\usepackage[dvipsnames]{xcolor}
\definecolor{orange}{rgb}{1.0, 0.5, 0.0}
\usepackage{fullpage}

\usepackage{mathtools}

\usepackage{thmtools}
\usepackage{thm-restate}
\usepackage{enumitem}
\usepackage{caption,subcaption}
\usepackage{hyperref}
\usepackage{cleveref}
\usepackage{soul}
\usepackage{float}
\usepackage[most]{tcolorbox}
\tcbuselibrary{listings, breakable}

\usepackage{graphicx}
\graphicspath{ {./images/} }

\declaretheorem[name=Theorem, numberwithin=section]{theorem}
\declaretheorem[name=Lemma, sibling=theorem]{lemma}
\declaretheorem[name=Proposition, sibling=theorem]{proposition}
\declaretheorem[name=Definition, sibling=theorem]{definition}

\declaretheorem[name=Question, sibling=theorem]{question}
\declaretheorem[name=Claim, sibling=theorem]{claim}
\declaretheorem[name=Claim, numbered=no]{claim*}

\declaretheorem[name=Remark, style=remark, sibling=theorem]{remark}

\renewcommand*{\theHtheorem}{\theHsection.\the\value{theorem}}

\newtcolorbox{problemBox}[1]{
  colback=white,
  colframe=black,
  coltitle=black,
  sharp corners,
  boxrule=0.8pt,
  toptitle=1mm,
  bottomtitle=1mm,
  enhanced,
  attach boxed title to top center={yshift=-3mm},
  boxed title style={
    colframe=white,
    colback=white,
    sharp corners,
    boxrule=0pt
  },
  title={#1} 
}

{}
{}
{}
{}

\newcommand{\pwr}[1]{\mathrm{pw}\left( #1 \right)}

\newcommand{\pw}{\mathrm{pw}}
\newcommand{\ns}{\mathrm{ns}}
\newcommand{\tw}{\mathrm{tw}}

\newcommand*\sg[1]{\{ #1 \}}
\newcommand{\cA}{\mathcal{A}}
\newcommand{\cB}{\mathcal{B}}

\newcommand{\cP}{\mathcal{P}}
\newcommand{\cQ}{\mathcal{Q}}

\newcommand{\tC}{\widetilde{C}}

\newcommand{\dist}{d}
\newcommand{\ecc}{\mathrm{ecc}}
\newcommand{\abs}[1]{\lvert #1 \rvert}

\title{A polynomial bound on the pathwidth of graphs edge-coverable by \texorpdfstring{$k$}{}  shortest paths}

\author[1]{Julien Baste}

\author[2]{Lucas De Meyer}

\author[3]{Ugo Giocanti\thanks{Supported by the National Science Center of Poland
under grant 2022/47/B/ST6/02837 within the OPUS 24 program, and partially supported by the French ANR
Project TWIN-WIDTH (ANR-21-CE48-0014-01).}}

\author[4]{Etienne Objois}

\author[5]{Timothé Picavet}

\affil[1]{University of Lille, CNRS, Centrale Lille, UMR 9189 CRIStAL, F-59000 Lille, France}
\affil[2]{Université Claude Bernard Lyon 1, CNRS, INSA Lyon, LIRIS, UMR 5205, 69622 Villeurbanne, France}
\affil[3]{Faculty of Mathematics and Computer Science, Jagiellonian University,
Kraków, Poland}
\affil[4]{IRIF, Université Paris Cité, France}
\affil[5]{LaBRI, Université de Bordeaux, France}

\date{}

\begin{document}
\maketitle

\begin{abstract}
Dumas, Foucaud, Perez and Todinca (2024) recently proved that every graph whose edges can be covered by $k$ shortest paths has pathwidth at most $O(3^k)$. In this paper, we improve this upper bound on the pathwidth to a polynomial one; namely, we show that every graph whose edge set can be covered by $k$ shortest paths has pathwidth $O(k^4)$, answering a question from the same paper. Moreover, we prove that when $k\leq 3$, every such graph has pathwidth at most $k$ (and this bound is tight). Finally, we show that even though there exist graphs with arbitrarily large treewidth whose vertex set can be covered by $2$ isometric trees, every graph whose set of edges can be covered by $2$ isometric trees has treewidth at most $2$.  
\end{abstract}
 
\vspace{.3cm}

\section{Introduction}
\label{sec: intro}

\emph{Graph covering} is an old and recurrent topic in graph theory. For a fixed class $\mathcal H$ of graphs, the covering literature aims to answer two types of problems. First are optimization problems: given a graph $G$, what is the minimum number of graphs from $\mathcal H$ needed to ``cover''\footnote{Various notions of coverings can usually be considered here.} $G$? We refer to Knauer and Ueckerdt~\cite{KU16} for a more general overview on the topic. The second type are structural problems: for the fixed class $\mathcal H$ of graphs and some $k\geq 1$, what is the structure of the graphs that can be ``covered'' by (at most) $k$ graphs from $\mathcal{H}$?

We say that a graph $G$ is \emph{edge-coverable} (resp. \emph{vertex-coverable}) by some subgraphs $H_1, \ldots, H_k$ if every \emph{edge} (resp. \emph{vertex}) of $G$ belongs to at least one of the subgraphs $H_i$.  
Motivated by some algorithmic applications to the \textsc{Isometric Path Cover} problem and its variants, Dumas, Foucaud, Perez and Todinca~\cite{DFPT24} recently worked on a particular metric version of such a covering problem. They proved that every graph edge-coverable (resp. vertex-coverable) by few shortest paths shares structural similarities with a simple path. More precisely, if a graph is edge-coverable or vertex-coverable by $k$ shortest paths, its \emph{pathwidth} is bounded by an exponential in $k$ (see \Cref{sec: Prelis} for a definition of pathwidth).

\begin{theorem}[\cite{DFPT24}]
 \label{thm: intro-expo}
 Let $k\geq 1$ and $G$ be a graph vertex-coverable by $k$ shortest paths. Then $\pw(G)=O(k\cdot 3^k)$. 
 Moreover, if $G$ is edge-coverable by $k$ shortest paths, then $\pw(G)=O(3^k)$. 
\end{theorem}

The proof of \Cref{thm: intro-expo} uses a branching approach to show that, in every graph vertex-coverable (resp. edge-coverable) by at most $k$ shortest paths, and for every vertex $u$ and every $i\geq1$, the number of vertices at distance exactly $i$ from $u$ in $G$ is $O(k\cdot3^k)$ (resp. $O(3^k)$). We also mention that a ``coarse'' variant of \Cref{thm: intro-expo} has recently been proved~\cite{hatzel2025graphs}. 
Then, the bags of the path decomposition proving \Cref{thm: intro-expo} are the union of every two consecutive layers in a BFS-layering of $G$ (starting from an arbitrary vertex). 
Such a path decomposition is clearly not optimal in general: if we let $S_h$ denote the star with $h$ leaves, then $\pw(S_h)=1$, while such a decomposition will always have width at least $\lfloor \tfrac{h}{2}\rfloor$, whatever the choice of $u$ is. 
Dumas et al.\ \cite{DFPT24} asked if the bounds from \Cref{thm: intro-expo} can be made polynomial in $k$. Our main result is a positive answer in the edge-coverable case.

\begin{restatable}{theorem}{mainPoly}
\label{thm: main}
 Let $k\geq 1$, and $G$ be a graph edge-coverable by $k$ shortest paths. Then $\pw(G)=O(k^4)$. 
\end{restatable}

Our proof uses a different approach from the one of  \cite{DFPT24}. Roughly speaking, it consists in finding some separator $X$ of polynomial size, such that for every component $C$ of $G-X$, $C$ has a ``nice'' layering with respect to one of the $k$ covering paths.

In terms of lower bounds, we provide for every $k\geq 1$, a construction of a graph edge-coverable by $k$ shortest paths and with pathwidth equal to $k$.

\begin{restatable}{proposition}{mainlb}
\label{prop:low_bound}
 For any integer $k\geq 1$, there exists a graph $G_k$ that is 
 \begin{enumerate}[label=(\alph*)]
  \item\label{enum:coverable} edge-coverable by $k$ shortest paths, and
  \item\label{enum:pathwidth} that has pathwidth at least $k$.
 \end{enumerate}
\end{restatable}

To our knowledge, no better lower bound is known. Our second main result is that this lower bound is tight for $k\leq 3$.
\begin{restatable}{theorem}{mainthreepaths}
\label{thm: 3-paths}
 Let $k\leq 3$, and $G$ be a graph edge-coverable by $k$ shortest paths. Then $\pw(G)\leq k$. 
\end{restatable}

Given a graph $G$, a subgraph $H$ of $G$ is \emph{isometric} if for every two vertices $x,y$ of $H$, the distances between $x$ and $y$ are the same in $G$ and $H$. In other words, if $x$ and $y$ are connected in $G$, then there must exist a shortest path between them in $G$ which is included in $H$. 

In view of the previous results, the following question naturally arises: given a family $\mathcal H$ of graphs, what is the structure of a graph $G$ edge/vertex-coverable by at most $k$ \emph{isometric} subgraphs all isomorphic to graphs in $\mathcal H$? For instance, if $\mathcal H$ is the family of all trees, does $G$ somehow look like a tree? 
   
Extending a result of Aigner and Fromme~\cite{AF84} for paths, Ball, Bell, Guzman, Hanson-Colvin and Schonsheck~\cite{BBGHCS17} proved that every graph vertex-coverable by $k$ isometric subtrees has \emph{cop number} at most $k$. In particular, it is well-known that the cop number of a graph is always upper-bounded by its \emph{treewidth} (see \Cref{sec: trees} for a definition of treewidth). 
In analogy to \Cref{thm: intro-expo}, a natural question is then the following: is it true that every graph which is vertex/edge-coverable by a small number of isometric subtrees has small treewidth? This question turns out to have a negative answer when considering vertex-coverability, as shown by the following example, pointed to us by Marcin Bria\'nski (personal communication): for every $t\geq 1$, let $G_t$ be the graph on $2t+2$ vertices obtained from the biclique $K_{t,t}$ after adding two vertices  $a,b$ of degree $t$ respectively adjacent to all vertices from each of the sides of $K_{t,t}$ (see \Cref{fig:vertex-coverable-trees} for the case $t=3$). 
Then, $G_t$ is vertex-coverable by the two induced stars $S_a, S_b$ respectively centered in $a$ and $b$ (which are in particular isometric subgraphs of $G_t$), but has treewidth at least $t$, as it contains $K_{t,t}$ as a minor. Our third result implies that any such construction fails when considering graphs which are edge-coverable by two isometric trees, as we show that they have treewidth at most $2$.

\begin{figure}
  \centering
  \includegraphics[scale=1.5]{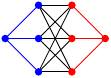}
  \caption{This graph is vertex-coverable by $2$ isometric trees (one with red vertices, the other with blue vertices) and contains $K_{3,3}$ as a minor (black edges). Hence, it has treewidth at least $3$.}
  \label{fig:vertex-coverable-trees}
\end{figure}

\begin{restatable}{theorem}{mainTrees}
\label{thm: 2-trees}
  Let $G$ be a graph which is edge-coverable by $2$ isometric trees. Then $\tw(G) \leq 2$.
\end{restatable}

\paragraph*{Algorithmic motivations.}
\emph{Routing problems} play a central role not only in algorithmic graph theory and complexity, but also in structural graph theory. One of the best examples is the \textsc{Disjoint Paths} problem, appearing in the seminal Graph Minor series of papers of Robertson and Seymour~\cite{RSXIII}, in which one is given as an input a graph $G$, an integer $k$ and $k$ pairs of vertices, and one must decide whether there exists $k$ pairwise disjoint paths connecting all pairs of vertices. This problem has been shown to be solvable in FPT time, when parameterized by the number $k$ of shortest paths~\cite{RSXIII, KR12}, and this result has been a key ingredient to design efficient algorithms deciding the existence of \emph{minors} in graphs. 
The metric variant \textsc{Disjoint Shortest Path} of this problem~\cite{Eilam98} (in which one further wants to decide whether there exist $k$ pairwise disjoint shortest paths connecting all input pairs of vertices) has attracted significant attention, and was recently proved to be solvable in XP time, when parameterized by $k$~\cite{BK17,DSPpoly,DSPpolygeo}.\footnote{It is worth mentioning that all the aforementioned problems are NP-hard~\cite{FHW80, DSPpoly}.} 

A number of metric variants of routing problems of similar flavor has been studied over the last decades. Among them, \textsc{Isometric Path Cover}~\cite{AF84, IPC} takes as input a graph $G$ and an integer $k$, and asks whether $G$ is vertex-coverable by $k$ shortest paths. Its routing version \textsc{Isometric Path Cover With Terminals} takes as input a graph $G$, some integer $k$, and $k$ pairs of vertices, and asks whether there exist $k$ shortest paths that vertex-cover $G$ and connect all the input pairs of vertices. Both problems are NP-complete~\cite{CDDFGG, DFPT24}. By combining \Cref{thm: intro-expo} with Courcelle's theorem~\cite{Courcelle}, it was shown~\cite{DFPT24} that there exists algorithms running in XP (respectively FPT time) when parameterized by $k$, to solve \textsc{Isometric Path Cover} (respectively \textsc{Isometric Path Cover With Terminals}). In particular, it is still open~\cite{DFPT24} whether the \textsc{Isometric Path Cover} problem admits an FPT algorithm parameterized by $k$. We also refer to Chakraborty et al.~\cite{CDFMT} for some recent work on similar metric versions of routing problems. It is worth mentioning that \textsc{Isometric Path Cover} 
admits some polynomial time $c$-approximation algorithms
on graphs with \emph{isometric path complexity} at most $c$~\cite{CCFV}, a recent metric parameter of graphs~\cite{CDDFGG} (introduced under the name of \emph{isometric path antichain cover number}).

As mentioned above, the original parameterized algorithms from~\cite{DFPT24} solving \textsc{Isometric Path Cover} and its variant with terminals rely on Courcelle's theorem, and thus the dependencies in $k$ are highly non-trivial and do not give rise in general to practical algorithms. Chakraborty et al.~\cite[Section 3]{CDFMT} designed a dynamic programming procedure solving the partition variant \textsc{Isometric Path Partition} 
of \textsc{Isometric Path Cover} in time $n^{O(\tw(G))}$, where $\tw(G)$ denotes the treewidth of the input graph. 
As pointed to us by Florent Foucaud (personal communication), it is very likely that combining our bound from \Cref{thm: main} with such an approach could lead to better running times than the ones given by Courcelle's theorem for algorithms solving \textsc{Isometric Path Edge-Cover} and its version with terminals (the natural edge-covering variant of \textsc{Isometric Path Cover}).
 
\paragraph*{Organization of the paper.}
\Cref{sec: Prelis} contains all preliminary definitions and some basic results related to pathwidth. Sections \ref{sec: poly}, \ref{sec: 3paths}, \ref{sec: trees}, and \ref{sec: lb} contain respectively proofs of Theorems \ref{thm: main}, \ref{prop:low_bound}, \ref{thm: 3-paths}, and \ref{thm: 2-trees}. Note that these sections are completely independent of each other, and can be read in any order. 
\Cref{sec: ccl} contains some discussion and additional questions left open by this work.

\section{Preliminaries}
\label{sec: Prelis}
\subsection{Notations and Basic Definitions}

\paragraph*{Graphs.}
In this work, we consider undirected finite simple graphs, without loops. For a graph $G$, we denote by $V(G)$ its set of \emph{vertices}, and $E(G)$ its set of \emph{edges}.
For any set $X \subseteq V(G)$ of vertices, $G[X]$ denotes the subgraph of $G$ induced by $X$, that is, the graph with vertex set $X$, and containing all edges of $G$ with both endpoints in $X$. $G-X$ denotes the subgraph of $G$  induced by $V(G) \setminus X$. For simplicity, for every vertex $x$, we let $G-x$ denote the graph $G-\sg{x}$.
For every subsets $X,Y\subseteq V(G)$ of vertices, we denote $E_G[X,Y]:=\sg{xy\in E(G): x\in X\setminus Y, y\in Y\setminus X}$. Let $G'$ be a graph, whose vertex set potentially intersects the one of $G$. The union of $G$ and $G'$ is the graph $G\cup G':=(V(G)\cup V(G'), E(G)\cup E(G'))$. 

For every $v\in V(G)$, we let $N_G(v)$ (or simply $N(v)$ when the context is clear) denote the set of neighbors of $v$ in $G$. Similarly, for every subset $X\subseteq V(G)$ of vertices, we let $N_G(X)$ (or simply $N(X)$) denote the set of vertices in $V(G)\setminus X$ having at least one neighbor in $X$. If $U\subseteq V(G)$, we similarly let $N_{G,U}(X)$ denote the set of vertices in $U\setminus X$ having at least one neighbor in $X$. When the context is clear, we will simply write $N_U(X)$. 

For every three subsets $X, A, B\subseteq V(G)$, we say that $X$ \emph{separates} $A$ from $B$ in $G$ if no component of $G-X$ contains both a vertex from $A$ and $B$.
 
\paragraph*{Paths.}
A path $P$ is a sequence of pairwise distinct vertices $v_1, \ldots, v_k$ such that for each $i<k$, $v_iv_{i+1}\in E(G)$. Note that by definition, a path is a sequence of vertices; however, by abuse of notations, we will often identify a path $P=v_1,\ldots, v_k$ of $G$ and its corresponding subgraph in $G$ (that is, the graph with vertices $v_1, \ldots, v_k$ and edges $v_iv_{i+1}$ for each $i$ such that $1 \leq i < k$). 
A path connecting two vertices $v_1=x$ and $v_k=y$ will also be called an \emph{$xy$-path}. We call the vertices $v_1$ and $v_k$ the \emph{endvertices} of $P$, while all its other vertices are its \emph{internal vertices}.
The \emph{length} of $P$ corresponds to its number of edges.
For each $i\leq j$, we let $P[v_i, v_j]$ denote the subpath $v_i, v_{i+1}, \ldots, v_j$ of $P$, and we let $P^{-1}$ denote the path $v_k,\ldots,v_1$. 
For every two paths $P=v_1,\ldots, v_p$ and $Q=w_1,\ldots, w_q$ such that $v_p=w_1$, and such that $P-v_p$ is disjoint from $Q-w_1$, we define the path $P\cdot Q:= v_1,\ldots, v_p=w_1, \ldots, w_q$. 
If $P=v_1,\ldots,v_k$ is a path, for every $1\leq i<j\leq k$, we say that $v_i$ is \emph{before} $v_j$ on $P$, and that $v_j$ is \emph{after} $v_i$ on $P$. 
A path $P$ is \emph{vertical} with respect to a vertex $u$, if for every $i\in \mathbb N$, $P$ has at most one vertex at distance $i$ from $u$ in $G$. Note that every path vertical with respect to some vertex $u$ is a shortest path, and that every shortest path starting from $u$ is vertical with respect to $u$.

\paragraph*{Covers.} Let $G$ be a graph and $H_1, \ldots, H_k$ be subgraphs of $G$. We say that the graphs $H_1, \ldots, H_k$ \emph{edge-cover} $G$ if $E(G)=\bigcup_{i=1}^kE(H_i)$. Similarly, we say that $H_1,\ldots, H_k$ \emph{vertex-cover} $G$ if $V(G)=\bigcup_{i=1}^kV(H_i)$. Note that whenever $G$ does not contain an isolated vertex, then every edge-cover is also a vertex-cover.

\paragraph*{Metrics.} We let $\dist_G(\cdot, \cdot)$ (or simply $d(\cdot, \cdot)$ when the context is clear) denote the shortest-path metric of a graph $G$. Recall that a graph $H$ is an \emph{isometric subgraph} of $G$ if $H$ is a subgraph of $G$ such that for every two vertices $x,y\in V(H)$, we have $d_H(x,y)=d_G(x,y)$.
For every vertex $v$ of a graph $G$, we let $\ecc(v):=\max\sg{d_G(u,v): u\in V(G)}\in \mathbb N\cup \sg{\infty}$ denote the \emph{eccentricity} of $v$.

\paragraph*{Orders.} Let $(X,\leq)$ be a partially ordered set. For every two disjoint subsets $A, B\subseteq X$, we write $A< B$ if $a<b$ for every $(a,b)\in A\times B$. Moreover, for every $x\in X$ and every $A\subseteq X$, we let $A|_{\leq x}:=\sg{a\in A: a\leq x}$ and $A|_{\geq x}:=\sg{a\in A: a\geq x}$. We define similarly $A|_{<x}$ and $A|_{>x}$. 
A \emph{chain} is a sequence $(x_1, \dots, x_k)$ of distinct elements of $X$ such that $x_1 \leq x_2 \leq \dots \leq x_k$.

\subsection{Path decompositions}
We collect in this section a number of basic properties of pathwidth.

\paragraph*{Path decomposition.}A path decomposition of a graph $G$ is a sequence $\mathcal{P}  = (X_1, X_2, \ldots , X_q)$ of vertex subsets of $G$, called bags, such that $V(G)=\bigcup_{i=1}^q X_i$,
for every edge $\lbrace x, y \rbrace \in E$ there is at least one bag containing both endpoints, and for every vertex $x \in V$, the bags containing $x$ form a continuous subsequence of $\mathcal{P}$. The width of $\mathcal{P}$ is $\max \lbrace \abs{X_i} -1 :1 \leq  i \leq  q \rbrace$, and the \emph{pathwidth} $\pw(G)$ of $G$ is the smallest width of a decomposition, among all path decompositions of $G$. The bags $X_1, X_q$ are called the \emph{extremities} of $\mathcal P$. We say that $X_1$ is the \emph{leftmost bag}, and that $X_q$ is the \emph{rightmost bag}.

\begin{remark}
\label{rem: X}
    For every graph $G$ and $X \subseteq V(G)$, $\pw(G) \leq \abs{X} + \pw(G - X)$.
\end{remark}

\paragraph*{Vertex-separation number.} Let $G$ be a graph and $<$ denote a total ordering on $V(G)$. The vertex-separation number of $(G,<)$ is the value $\min_{w\in V(G)}|\sg{u \in V(G): u<w ~\text{and}~\exists v\geq w, uv\in E(G)}|$. The vertex separation number of $G$ is the minimum over the vertex separation numbers of $(G,<)$, for all possible linear orderings $<$ of $G$.

\begin{theorem}[\cite{Kinnersley}]
\label{thm: Kinnersley}
    The vertex separation number of a graph equals its pathwidth.  
\end{theorem}

\paragraph*{Layerings.} A \emph{layering} of a graph $G$ is a partition of its vertices into a sequence $(V_1, \ldots, V_t)$ of subsets of $V(G)$ called \emph{layers}, such that each graph $G[V_i]$ is an independent set, and for every $i,j\in \sg{1, \ldots, t}$ such that $|j-i|\geq 2$, we have $E_G(V_i, V_j) = \emptyset$.\footnote{Note that our definition of layering differs from the one which is commonly used in the literature, as one usually allows edges between vertices from a same layer.}

A \emph{$k$-layering} of $G$ is a layering $(V_1, \ldots, V_t)$ such that for every $0 < i<t$, we have $\abs{E_G(V_i, V_{i+1})}\leq k$. Note that in particular, it implies that $\abs{N_{V_{i+1}}(V_i)} \leq k$ and $\abs{N_{V_{i}}(V_{i+1})} \leq k$. The bags $V_1$ and $V_t$ are the \emph{extremities} of the layering. 

\begin{lemma}
\label{lem: k-layered}
    Let $G$ be a graph admitting a $k$-layering. Then $G$ has pathwidth at most $k$. Moreover, if $G$ is connected, then $G$ has a path decomposition of width at most $k$, whose extremities are exactly the extremities of the $k$-layering.
\end{lemma}

\begin{proof}
We let $(V_1, \ldots, V_t)$ denote a $k$-layering of $G$, and assume without loss of generality that $G$ has no isolated vertex.
We also set $V_0:=V_{t+1}:=\emptyset$.
For every $i\in \sg{1, \ldots, t}$, we set $V_i^{\ell}:=N_{V_{i}}(V_{i-1})$ and $V_i^{r}:=N_{V_i}(V_{i+1})$. Note in particular that, as $G$ has no isolated vertex, for each $i\in \sg{1,\ldots, t}$ we have $V_i=V_i^{\ell}\cup V_i^r$, $|V_i^{\ell}|\leq k$, $|V_i^r|\leq k$, and $V_1^{\ell}=V_t^{r}=\emptyset$, and that for each $i<t$, all the edges of $G$ between $V_i$ and $V_{i+1}$ are between 
$V_i^{r}$ and $V_{i+1}^{\ell}$. 

We will define a total order $\leq$ on $V(G)$ such that $(G,\leq)$ has vertex separation number at most $k$. In particular, \Cref{thm: Kinnersley} will then immediately imply that $\pw(G)\leq k$. We let $\leq$ be any total ordering on $V(G)$ satisfying the following two properties: 
\begin{itemize}
 \item for every $i<j$, we have $V_i<V_j$;
 \item for every $i\in \sg{1,\ldots, t}$, we have $V^{\ell}_i\setminus V^{r}_i<V^{\ell}_i\cap V^{r}_i<V^{r}_i\setminus V^{\ell}_i$.
\end{itemize}
We now show that $(G,\leq)$ has vertex separation number at most $k$ and let $w\in V(G)$. We set $S_w:=\sg{u: u<w \text{ and }\exists v\geq w: uv\in E(G)}$. Our goal is to show that $|S_w|\leq k$. We let $i$ denote the unique element from $\sg{1,\ldots, t}$ such that $w\in V_i$. 
We define the sets $A:=V^{r}_i|_{<w}, B:=V^{\ell}_i|_{\geq w}$ and $C:=V^{r}_{i-1}$. Note that $A,B,C$ are pairwise disjoint, and that moreover, as $(V_1,\ldots, V_t)$ is a layering of $G$, $A\subseteq S_w\subseteq A\uplus C$.

First, observe that if $B$ is empty, then we must have $S_w\cap C=\emptyset$, so that $S_w=A$, and thus as every vertex of $A$ is by definition of $V^{r}_i$ adjacent to a vertex of $V_{i+1}$, and as $|E_G(V_i,V_{i+1})|\leq k$, we then have $|S_w|\leq k$. Assume now that $B$ is not empty. Then there exists an element $v\geq w$ such that $v\in V^{\ell}_i$. Thus, by definition of $<$, we have $V_i|_{\leq w}\subseteq V^{\ell}_i$, and thus $A\subseteq V^{\ell}_i$.

We now set $k_1:=|E_G(C,A)|$ and $k_2:=|E_G(C,B)|$. As $A\subseteq V^{\ell}_i$, every element of $A$ must be incident to an edge from $E_G(A,C)$, hence $|A|\leq k_1$. As $C=V_{i-1}^r$, we have $|C|\leq k_2$.
Moreover, note as $(V_1, \ldots, V_t)$ is a $k$-layering and as $A$ and $B$ are disjoint, we have $k_1+k_2\leq k$. In particular, it implies that 
$$|S_w|\leq |A|+|C|\leq k_1+k_2\leq k,$$
hence we can conclude that $(G,\leq)$ has indeed vertex separation number at most $k$.

To prove the ``Moreover'' part of the lemma, observe that if $G$ is connected, then both $V_1$ and $V_t$ have size at most $k$. We claim that it is not hard to check that the path decomposition given by \Cref{thm: Kinnersley} can then be slightly modified so that its extremities are $V_1$ and $V_k$. 
\end{proof}

\section{A General Polynomial Bound}
\label{sec: poly}
In this section, we prove \Cref{thm: main}, which we restate here.  

\mainPoly*

Throughout this section, we let $G$ be a graph whose edge set can be covered by a set $\cP:=\sg{P_1, \ldots, P_k}$ of $k$ isometric paths. Our main result is the following, which will immediately imply a proof of \Cref{thm: main}.

\begin{theorem}\label{thm: main-poly}
For every $i \in \sg{1,\ldots,k}$, there exists a set of vertices $X_i \subseteq V$ of size at most $720k^3+4k$ such that every connected component of $G-X_i$ that intersects $P_i$ has pathwidth at most $6k$.
\end{theorem}

\begin{proof}[Proof that \Cref{thm: main-poly} implies \Cref{thm: main}]
For every $1\leq i\leq k$, we use \Cref{thm: main-poly} to obtain the set $X_i$. We consider their union $X:=\bigcup_{i=1}^kX_i$. Observe that $\abs{X} \leq \sum_{i\in [k]} \abs{X_i} \leq k \cdot (720 k^3 + 4k) = 720k^4 + 4k^2$.
We now let $C$ be a connected component of $G - X$. As $P_1, \ldots, P_k$ edge-cover (and thus also vertex-cover) $G$, there is some $i\in \sg{1,\ldots,k}$ such that the path $P_i$ intersects $C$. In particular, as $C$ is included in a connected component of $G-X_i$, it implies that $C$ has pathwidth at most $6k$, and thus that $\pw(G - X) \leq 6k$. 
Using \Cref{rem: X} we conclude that $\pw(G) \leq \abs{X} + \pw(G -X) \leq 720k^4+4k^2+6k$.
\end{proof}

The remainder of the section will consist in a proof of \Cref{thm: main-poly}. In Section \ref{sec: parallel} we prove some preliminary results, and show that our proof reduces to the problem of finding a special kind of separations in a subgraph $G_0$ of $G$ which is $6k$-layered (see \Cref{lem: keyP1}). In Section \ref{sec: key}, we give a proof of \Cref{lem: keyP1}, which is the most technical part of our proof.

Without loss of generality, we will assume from now on that $i=1$, i.e., we will construct a set $X$ of size at most $720k^3+4k$ such that every component of $G-X$ that intersects $P_1$ has pathwidth at most $6k$.

\subsection{Preliminary Results}
\label{sec: parallel}

\paragraph*{Parallel paths.} Let $P$ be a shortest path between two vertices $a,b$ of $G$. A path $Q$ of $G$ is \emph{parallel} to $P$ in $G$ if $Q$ is a subpath of a shortest $ab$-path of $G$. 

Note that this does not define a symmetric relation in general: the fact that a path $Q$ is parallel to another path $P$ does not necessarily imply that $P$ is parallel to $Q$. However, in the remainder of the proof, we will always consider the property of being parallel to the path $P_1$, which is fixed and connects vertices $a_1$ and $b_1$. In particular, the following simple observation implies that every graph which is edge-coverable by a few number of paths that are all parallel to $P_1$ has a very restrained structure. 

 \begin{lemma}
  \label{clm: parallel implies layered}
  Let $\sg{Q_1, \ldots, Q_{\ell}}$ be a collection of paths in $G$ that are all parallel to $P_1$. Then the graph $\bigcup_{j=1}^{\ell}Q_j$ is $\ell$-layered. 
 \end{lemma}

 \begin{proof}
 We set $U:=\bigcup_{i=1}^{\ell}V(Q_i)$, $r:=d(a_1, b_1)$, and for every $i\in \sg{0,\ldots, r}$, we let $U_i:=\sg{u\in U: d(a_1, u)=i}$. As every path $Q_j$ is parallel to $P_1$, observe that for each $(i, j)\in \sg{0,\ldots, r}\times\sg{1,\ldots, \ell}$, we must have $|V(Q_j)\cap U_i|\leq 1$. In particular, for each $i\in \sg{0,\ldots, r}$, we have $|U_i|\leq \ell$. Hence, $(U_0, \ldots, U_r)$ forms a partition of $U$, such that each edge of $G[U]$ has its two endvertices in some pair $(U_i, U_{i'})$ such that $|i'-i|\leq 1$. Moreover, as the paths $Q_j$ are all parallel to $P_1$, no edge of a path $Q_j$ can have both endvertices in the same set $U_i$, so $(U_0,\ldots, U_r)$ is indeed a $\ell$-layering of $\bigcup_{j=1}^{\ell}Q_j$.
 \end{proof}

 \begin{lemma}
  \label{clm: parallel-hereditary}
  If $P$ is a path parallel to $P_1$ with endvertices $u,v\in V(G)$, then every shortest $uv$-path in $G$ is also parallel to $P_1$.
 \end{lemma}
 
 \begin{proof}
  This immediately follows from the observation that every subpath of a shortest path is also a shortest path. 
 \end{proof}
 
 We consider the partial order $\leq_1$ on $V(G)$ defined by setting for every two vertices $u,v\in V(G)$, 
 $u\leq_1 v$ if and only if $d_G(a_1,u)\leq d_G(a_1,v)$. Note in particular that the vertex set of every path $P$ which is parallel to $P_1$, forms a chain with respect to $\leq_1$.

 \begin{lemma}
  \label{lem: parallel-concatenation}
  Let $a\leq_1 b\leq_1 c$ be three vertices such that there exist two paths $P,Q$ both parallel to $P_1$, such that $P$ connects $a$ to $b$ and $Q$ connects $b$ to $c$. Then $P\cdot Q$ is parallel to $P_1$. 
 \end{lemma}

 \begin{proof}
  Note that $P\cdot Q$ must be a shortest $ac$ path, as for each $i\geq 0$, it contains at most one vertex at distance exactly $i$ from $a_1$. In particular, it is not hard to conclude that $P\cdot Q$ is parallel to $P_1$.
 \end{proof}
 
 \paragraph*{Reducing $\cP$.}
 The first step in the proof of \Cref{thm: main-poly} consists in modifying the covering family $\cP$ in order to make it reduced (see definition below). The intuition behind this step is that our proof of \Cref{thm: main-poly} works particularly well when all paths in $\cP$ intersect $P_1$. Informally, the reduction operation consists in making $\cP$ contain as much paths that intersect $P_1$ as possible. This step will turn out to be useful at the very end of our proof (see the proof of \Cref{lem: QinP2}).
 
 A shortest path $Q$ of $G$ is called a \emph{reducing path} if: 
  \begin{itemize}
  \item $Q$ is disjoint from $P_1$, and admits a subpath $P$ parallel to $P_1$, that connects two vertices $u, v$, such that $u<_{1}v$,
  \item there exists a shortest path $R$ in $G$ from $u$ to $v$ that intersects $P_1$.
 \end{itemize}
\Cref{fig: transformP2} depicts a reducing path. If a family $\cQ$ of shortest paths has no reducing path, then we say that $\cQ$ is \emph{reduced}.
 
 \begin{proposition}
  \label{prop: reduced}
  Let $\cQ$ be a family of shortest paths. Then there exists a family $\cQ'$ of shortest paths which is reduced, such that $|\cQ'|\leq 2|\cQ|$ and that the paths of $\cQ'$ cover all the edges covered by the paths of $\cQ$, i.e., such that
  $$\bigcup_{P\in \cQ}E(P)\subseteq \bigcup_{P\in \cQ'}E(P).$$
 \end{proposition}
 \begin{proof}
  Assume that $\cQ$ is not reduced, and let $Q$ be a reducing path in $\cQ$. We also let $P$ be a subpath of $Q$ with endvertices $u,v$ such that $u<_1v$, and $R$ be a shortest path in $G$ from $u$ to $v$ intersecting $P_1$. We let $R_1, R_2$ denote two shortest paths respectively from $a_1$ to $u$ and from $v$ to $b_1$. We also write $Q=Q_1\cdot Q_2\cdot Q_3$, so that $Q_2$ is the $uv$-subpath of $Q$ parallel to $P_1$ (see \Cref{fig: transformP2}). Observe now that both paths $Q':=Q_1\cdot R\cdot Q_3$ and $Q'':=R_1\cdot Q_2\cdot R_2$ are shortest paths in $G$ that both intersect $P_1$ and cover all the edges of $Q$. In particular, it implies that the family of paths $\cQ':=(\cQ\setminus \sg{Q})\cup\sg{Q', Q''}$ covers all the edges covered by the paths of $\cQ$, and has strictly fewer paths than $\cQ$ that do not intersect $P_1$. 
  We apply iteratively the same operation as long as the obtained family $\cQ'$ is not reduced. In particular, as at each step, the number of paths of $\cQ$ which are disjoint with $P_1$ strictly decreases, then after at most $|\cQ|$ iterations, we obtain a reduced family with the desired properties.
 \end{proof}

\begin{figure}[htb]
  \centering  
  \includegraphics[scale=1]{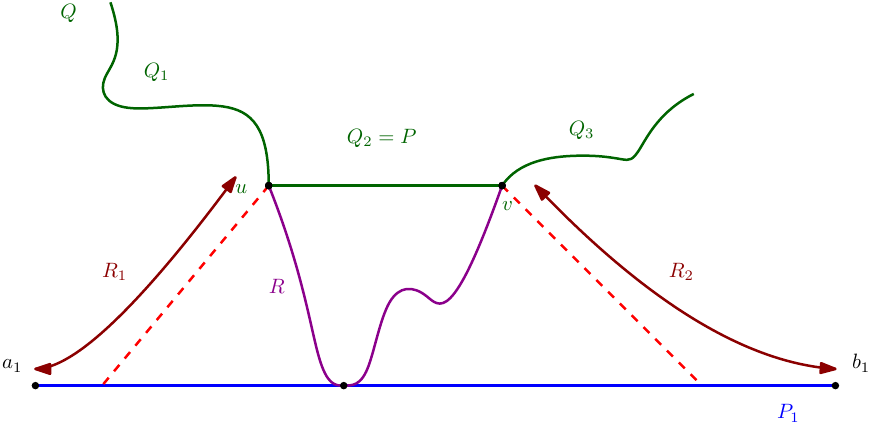}
  \caption{The path $Q$ represented in green is reducing. $P_1$ is represented in blue.}  
  \label{fig: transformP2}
\end{figure}

Up to applying \Cref{prop: reduced} to $\cP$, we will assume from now on that $\cP$ is a reduced family of at most $2k$ shortest paths of $G$ that edge-cover $G$. 

\paragraph*{Good and bad vertices.}
We now partition the set $\cP:=\sg{P_i: 1\leq i\leq 2k}$ into the subset $\cP_1$ of the paths in $\cP$ that intersect $P_1$, and $\cP_2:=\cP\setminus \cP_1$. 
 
For every $P\in \cP_1$, we let $C_P$ denote the maximal subpath of $P$ with both endvertices $a_P, b_P$ in $V(P_1)$, so that $a_P\leq_1 b_P$ (up to considering $P^{-1}$ instead of $P$, we also assume that $a_P$ is before $b_P$ on $P$). Note that $C_{P_1}=P_1$, and that by definition of $\cP_1$, each $C_P$ has at least one vertex, and is parallel to $P_1$.
We moreover consider for every $P\in \cP_1$ the two vertices $u_P, v_P$ that respectively maximize $d_G(u_P, a_P)$ and $d_G(b_P, v_P)$, such that $u_P,a_P,b_P,v_P$ appear in this order on $P$, and such that the paths $P[u_P, a_P]$ and $P[b_P, v_P]$ are parallel to $P_1$. We now let $\tC_P:=P[u_P, v_P]$ and 
$A_P, B_P$ denote the two subpaths of $P$ such that $P=A_P\cdot \tC_P\cdot B_P$ (see \Cref{fig: AiBiCi}). We stress out that in general, the path $\tC_P$ is not necessarily parallel to $P_1$.  
 
\begin{figure}[htb]
  \centering  
  \includegraphics[scale=1]{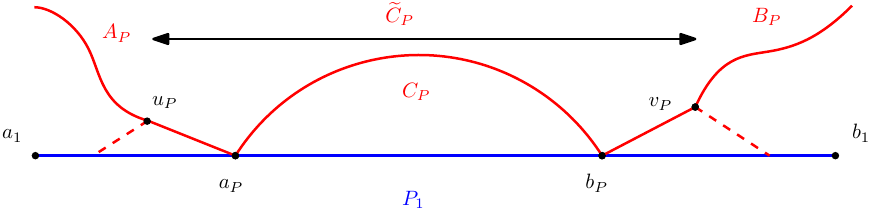}
  \caption{The red path represents some path $P\in \cP_1$. The blue path is $P_1$.}  
  \label{fig: AiBiCi}
\end{figure}

 We let $\cA:= \sg{A^{-1}_P: P\in \cP_1}\cup \sg{B_P: P\in \cP_1}$ and $\cB:=\sg{\tC_P: P\in \cP_1}$. In particular, $P_1\in \cB$. Note that by definition, every path in $\cB$ is the concatenation of at most $3$ paths that are parallel to $P_1$, and recall that $|\cP|\leq 2k$, hence we have $|\cB|\leq 2k$. 

 We say that the edges of the paths in $\cA\cup\cP_2$ are \emph{bad}, and we call a 
 vertex $v\in V(G)$ \emph{bad} if it is the endvertex of a bad edge. We call every vertex which is not bad \emph{good}, and we let $V_b$ and $V_g$ denote respectively the set of bad vertices and the set of good vertices of $G$. We moreover set $X_0:=\sg{a_P: P\in \cP_1}\cup \sg{b_P: P\in \cP_1}$ and $X_1:=\sg{u_P: P\in \cP_1}\cup \sg{v_P: P\in \cP_1}$. 
 Note that in particular, $|X_0|\leq 4k$ and that the only bad vertices that might belong to $P_1$ are the vertices $a_P, b_P$ for $P\in \cP_1$ whenever they are respectively equal to $u_P, v_P$, hence $V(P_1)\cap V_b\subseteq X_0$.  

 We now consider the subgraph $G_0 := \bigcup_{P\in \cB}P$ of $G$.
 The following remark immediately follows from the definition of $\cB$, the fact that $|\cP_1|\leq |\cP|\leq 2k$, and that every path in $\cB$ is edge-covered by at most $3$ paths that are parallel to $P_1$.

 \begin{remark}
  \label{rem: G0}
  $G_0$ is edge-covered by at most $6k$ paths that are parallel to $P_1$.
 \end{remark}

 \begin{lemma}
  \label{clm: Good}
  We have $V_g\subseteq V(G_0)$. Moreover, $G[V_g]$ is a subgraph of $G_0$.  \end{lemma}
 \begin{proof}
  Let $u$ be a good vertex. As $G$ is covered by $\cP$, there exists some path $P\in \cP$ such that $u\in V(P)$. In particular, as $u$ is a good vertex, we must have $P\in \cP_1$, and moreover $u$ must be a vertex of $\tC_{P}$. It implies that $u\in V(G_0)$, and thus that $V_g\subseteq V(G_0)$.
  The second part of the lemma follows from the fact that by definition, no bad edge can be incident to a good vertex. In particular, note that every edge from $E(G)\setminus E(G_0)$ is bad, hence $G[V_g]$ is indeed a subgraph of $G_0$.
 \end{proof}

 Our goal from now on will be to prove the following lemma. 
 
 \begin{lemma}[Main]
 \label{lem: mainP1}
  There exists a set $X\subseteq V(G)$ of size at most $720k^3+4k$ that separates every vertex on $P_1$ from $V_b$ in $G$.
 \end{lemma}

\begin{proof}[Proof of \Cref{thm: main-poly}~using \Cref{lem: mainP1}]
We let $X$ be given by \Cref{lem: mainP1}. In particular, by \Cref{clm: Good}, every component of $G-X$ that contains a vertex of $P_1$ induces a subgraph of $G_0$. By \Cref{lem: k-layered}, \Cref{clm: parallel implies layered} and \Cref{rem: G0}, $G_0$ is $6k$-layered, and thus has pathwidth at most $6k$, allowing us to conclude the proof of \Cref{thm: main-poly}.
By \Cref{rem: G0}, $G_0$ is edge-covered by at most $6k$ paths parallel to $P_1$, which by \Cref{lem: parallel-concatenation} implies that it has a $6k$-layering, and thus by \Cref{lem: k-layered}, its pathwidth is at most $6k$. 
\end{proof}

\paragraph*{Reduction to a separation problem in the graph \texorpdfstring{$G_0$}{}.}
 We will now show that, in order to prove \Cref{lem: mainP1}, it is enough to separate $V(P_1)$ from $V_b$ in the subgraph $G_0$ of $G$. 
We show that \Cref{lem: mainP1} amounts to prove the following key lemma, that we will prove in \Cref{sec: key}.

 \begin{lemma}[Key lemma]
 \label{lem: keyP1}
  Let $P$ be a path of $G_0$ which is parallel to $P_1$, and $A\in \cA\cup \cP_2$ which intersects $P$. 
  Then there exists some set $X_{A,P} \subseteq V(G_0)$ of vertices of size $30k$ such that $X_{A,P}$ separates every vertex of $V(A)\cap V(P)$ from the vertices of $P_1$ in $G_0$.
 \end{lemma}

\begin{figure}[htb]
  \centering  
  \includegraphics[scale=1]{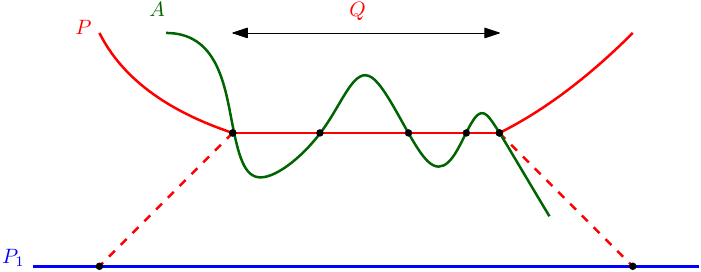}
  \caption{Configuration of \Cref{lem: keyP1}.}  
  \label{fig: keyL}
\end{figure}

 \begin{proof}[Proof of \Cref{lem: mainP1} using \Cref{lem: keyP1}]
  Recall that every path of $\cB$ is the concatenation of at most $3$ paths parallel to $P_1$. We thus consider a set $\cB'$ of size at most $3|\cB|$ of paths parallel to $P_1$ that cover the edges of the paths from $\cB$. 
  We let $I$ denote the set of all pairs $(A,P)$ such that $A\in \cA\cup \cP_2$, $P\in \cB'$, and $V(A)\cap V(P)\neq \emptyset$.
  
  We set 
  $$X:=X_0\cup \left(\bigcup_{(A,P)\in I}X_{A,P}\right),$$
  where the sets $X_{A,P}$ are given by \Cref{lem: keyP1}.
  Note that $|\cA\cup \cP_2| \leq 2|\cP_1|+ |\cP_2| \leq |\cP_1| + |\cP| \leq 4k$ and $|\cB'|\leq 3|\cB|=3|\cP_1|\leq 6k$, thus $|I|\leq 24k^2$ and $|X|\leq 720k^3+4k$.

  We now show that $X$ separates $V(P_1)$ from $V_b$ in $G$. Suppose, for contradiction, that there exists a path $R$ in $G-X$ connecting a vertex $v$ from $V_b$ to a vertex $u$ on $P_1$. Among those paths, choose a path  $R$ with minimum length. Since $V(P_1)\cap V_b\subseteq X_0\subseteq X$, we have $v \notin V(P_1)$, thus $R$ has length at least $1$. Also, since $R$ is a path of minimum length connecting $V_b$ to $P_1$, every vertex in $R-v$ must be good. Since every vertex in $R-v$ is good, there is no bad edge in $R$. Hence, as $\cP$ edge-covers $G$, all the edges of $R$ must be covered by the paths from $\cB$, implying that $R$ is also a path of $G_0$.
  
 This implies in particular that $v\in V(G_0)\cap V_b$. Thus by definition of $G_0$ and $V_b$, there exist some paths $P\in \cB$ and $A\in \cA\cup \cP_2$ such that $v\in V(P)\cap V(A)$. 
 \Cref{lem: keyP1} then implies that $R$ intersects $X_{A,P}\subseteq X$, giving a contradiction.
\end{proof} 

\subsection{Proof of \texorpdfstring{\Cref{lem: keyP1}}{Lemma~\ref{lem: keyP1}}}
\label{sec: key}
We fix some path $P$ in $G_0$ which is parallel to $P_1$, and some path $A\in \cA\cup \cP_2$ which intersects $P$. We let $Q$ denote the minimal subpath of $P$ containing all vertices in $V(A)\cap V(P)$, and let $b,c$ denote the endvertices of $Q$ such that $b\leq_1 c$. We will in fact prove that one can find a set $X_{A,P}$ of vertices with the desired size that separates in $G_0$ all the vertices of $Q$ from $P_1$.
By \Cref{rem: G0}, $G_0$ is coverable by at most $6k$ paths which are all parallel to $P_1$. In particular, for every $i \geq 0$, the set $V_i:=\sg{v\in V(G_0): d_G(v,a_1)=i}$ contains at most $6k$ vertices, and separates in $G_0$ the sets $\bigcup_{j\leq i} V_j$ and $\bigcup_{j\geq i}V_j$.
We also may assume without loss of generality that $Q$ contains at least $2$ vertices, hence $b<_1c$ (otherwise, we conclude by choosing $X_{A,P} := V(Q) = \sg{b}$).

For every $v\in V(G_0)$, we let $\iota(v)$ denote the unique index $i\geq 0$ such that $v\in V_{i}$. For every path $L$ in $G_0$ which is parallel to $P_1$, if $u,v$ denote the endvertices of $L$ with $u\leq_1v$, we define the \emph{projection} of $L$ on $P_1$ as the set $U_L:=\sg{x\in V(P_1): \iota(u)<\iota(x)<\iota(v)}$. We say that $L$ is \emph{$\nearrow$-free} if there does not exist in $G_0$ a path $R$ parallel to $P_1$ that connects a vertex $x\in U_L$ to a vertex $y\in V(L)$ with $x\leq_1 y$ (see \Cref{fig: fleche-free}). Symmetrically, we say that $L$ is \emph{$\searrow$-free} if there does not exist in $G_0$ a path $R$ parallel to $P_1$ that connects a vertex $x\in U_L$ to a vertex $y\in V(L)$ with $x\geq_1 y$. 
Our main result in this subsection will be \Cref{lem: left-to-right}, which states that if $Q$ is $\nearrow$-free or $\searrow$-free, then we can separate $Q$ from $P_1$ in $G_0$ after removing $30k$ vertices. Before proving this, we will first show in the next two lemmas that $Q$ can be covered by at most two subpaths which are each $\nearrow$-free or $\searrow$-free.

\begin{figure}[htb]
  \centering  
  \includegraphics[scale=0.8]{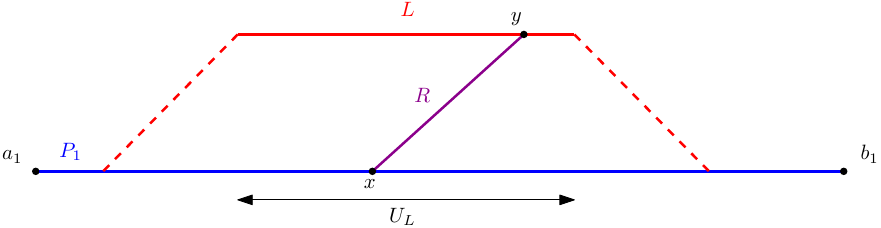}
  \caption{The path $L$ (in red) is not $\nearrow$-free.}  
  \label{fig: fleche-free}
\end{figure}

\begin{lemma}
\label{lem: QinA}
 If $A\in \cA$, then $Q$ is $\nearrow$-free or $\searrow$-free.
\end{lemma}
 
\begin{proof}
We assume first that $A=A^{-1}_S$ for some $S\in \cP_1$, and claim that the case where $A=B_S$ is symmetric. 
We moreover assume that $A[u_S, b]$ is a subpath of $A[u_S, c]$, as depicted in \Cref{fig: left-to-right}, and claim that the case where $A[u_S, c]$ is a subpath of $A[u_S, b]$ is symmetric. 
We show that in this case, $Q$ must be $\nearrow$-free. 
As $P$, and thus $Q$ are parallel to $P_1$, note that every vertex $v\in V(Q)$ satisfies $\iota(v)\in [\iota(b),\iota(c)]$. We let $A'$ denote the subpath of $S^{-1}$ which starts at $a_S$ and contains $A$ as a subpath.

We assume for sake of contradiction that there exists some path $R$ in $G_0$ which is parallel to $P_1$ in $G$, and which connects a vertex $x\in U_Q$ to a vertex $y\in V(Q)$, with $x\leq_1y$.  
We consider two cases according to whether $a_S\leq_1 x$ or not, described in \Cref{fig: left-to-right}.

If $a_S\leq_1 x$, then by \Cref{lem: parallel-concatenation}, the path $P_1[a_S,x]\cdot R\cdot Q[y,c]$ is parallel to $P_1$. On the other hand, recall that  by definition of $u_S$, and as $c\neq u_S$ (this is because $b\neq c$ and $A[u_S, c]$ contains $b$), $A'[a_S, c]$ cannot be parallel to $P_1$.
In particular, $A'[a_S,c]$ must be strictly longer than $P_1[a_S,x]\cdot R\cdot Q[y,c]$, implying a contradiction as $A'$ is a shortest path in $G$.

If $a_S>_1 x$, then as $x\in U_Q$, $x>_1 b$. 
Then, the path $R\cdot Q[y,c]$ is vertical with respect to $a_1$, and thus strictly shorter than $Q$. Note in this case, the path $A'[a_S, b]$ is also strictly longer than $P_1[a_S, x]$. It thus implies that $P_1[a_S, x]\cdot R\cdot Q[y,c]$ is strictly shorter than the path $A'[a_S, c]$, contradicting again that $A'$ is a shortest path in $G$.
\end{proof}

\begin{figure}[htb]
  \centering  
  \includegraphics[scale=0.8, page=1]{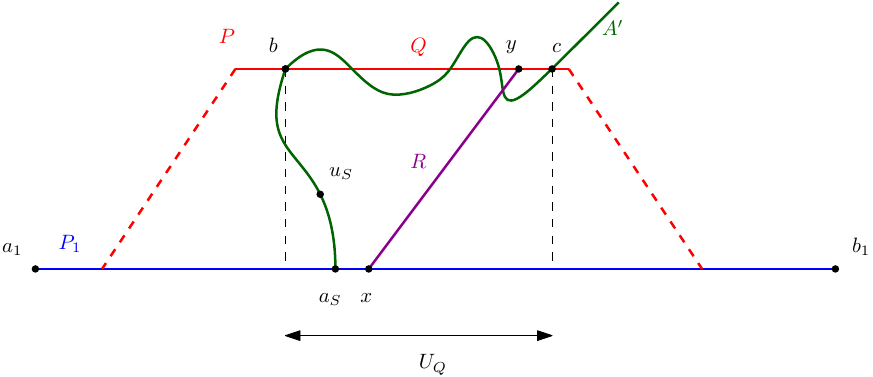}
  \includegraphics[scale=0.8, page=2]{left-to-right}
  \caption{Top: configuration in the proof of \Cref{lem: QinA} when $a_S\leq_1 x$. Bottom: configuration in the proof of \Cref{lem: QinA} when $a_S>_1 x$.}  
  \label{fig: left-to-right}
\end{figure}

\begin{lemma}
\label{lem: QinP2}
 If $A\in \cP_2$, then there exists subpaths $Q_1, Q_2$ such that $Q=Q_1\cdot Q_2$, $Q_1$ is $\searrow$-free and $Q_2$ is $\nearrow$-free.
\end{lemma}

\begin{proof}
 If $Q$ is $\nearrow$-free, then we immediately conclude by choosing $Q_1$ to be the path having $b$ as a single vertex, thus we may assume that there exists some path $R$ parallel to $P_1$ connecting some vertex $x\in U_Q$ to a vertex $y\in V(Q)$ such that $x\leq_1 y$. We moreover choose such a path so that $x$ is maximal with respect to $\leq_1$. We let $z$ be the unique vertex of $Q$ such that $\iota(z)=\iota(x)$, and set $Q_1:=Q[b,z], Q_2:=Q[z,c]$.
 By maximality of $x$, note that $Q_2$ must be $\nearrow$-free. We moreover claim that $Q_1$ is $\searrow$-free. Assume for a contradiction that it is not the case, and that there exists a shortest path $R'$ parallel to $P_1$ connecting a vertices $y'\leq_1 x'$, with $y'\in V(Q_1)$ and $x'\in U_{Q_1}$ (see \Cref{fig: down-then-up}). Then the path $T:=Q_1[b,y']\cdot R'\cdot P_1[x',x]\cdot R\cdot Q_2[y,c]$ is a path in $G$ from $b$ to $c$ that intersects $P_1$ and which is parallel to $P_1$. In particular, as $b,c\in V(A)$ the existence of $T$ implies that $A$ is a reducing path belonging to $\cP$, contradicting our assumption that $\cP$ is reduced.
\end{proof}

\begin{figure}[htb]
  \centering  
  \includegraphics[scale=0.8]{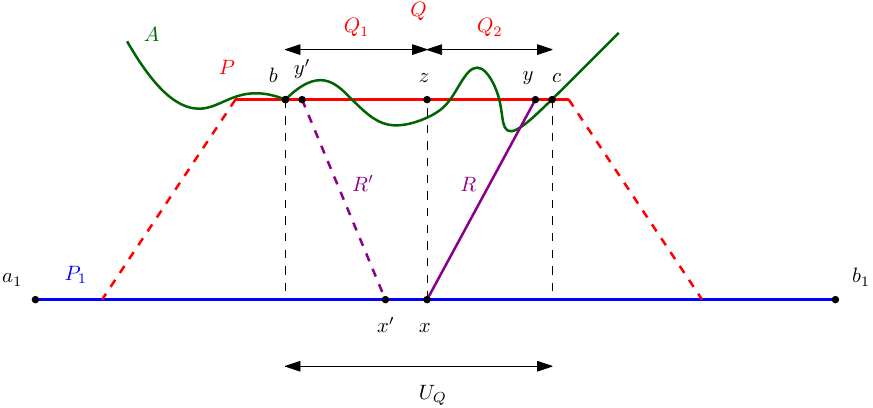}
  \caption{The configuration at the end of the proof of \Cref{lem: QinP2}. The purple dashed path $R'$ represents a shortest path which is parallel to $P_1$, and connects vertices $y'\leq_1 x'$, with $y'\in V(Q_1)$ and $x'\in U_{Q_1}$.}  
  \label{fig: down-then-up}
\end{figure}

The next lemma is the crucial part from the proof of \Cref{lem: keyP1}.

\begin{lemma}
 \label{lem: left-to-right}
 Let $L$ be a path in $G_0$ that is parallel to $P_1$. If $L$ is $\nearrow$-free or $\searrow$-free, then there exists $X\subseteq V(G_0)$ of size at most $18k$ that separates in $G_0$ the vertices of $L$ from the ones of $P_1$.
\end{lemma}

\begin{proof}
 We let $b_L, c_L$ denote the endvertices of $L$, so that $b_L\leq_1 c_L$, and let $X'_0:=V_{\iota(b_L)}\cup V_{\iota(c_L)}$. Recall that by \Cref{rem: G0}, 
 $G_0$ is edge-coverable by $k'$ paths $Q_1, \ldots, Q_{k'}$, all of them parallel to $P_1$ for some $k'\leq 6k$. 
 In particular, it implies that $X'_0$ has size at most $12k$, and separates $L$ from every vertex $x$ of $P_1$ such that $x\leq_1 b$ or $x\geq_1c$. It remains to separate $L$ from $U_L$. We assume without loss of generality that $L$ is $\nearrow$-free, the other case being symmetric, up to exchanging the roles of $a_1$ and $b_1$. We set $Q_0:=L$, and for each $j\in \sg{0,\ldots, k'}$, we let $H_j:= P_1\cup Q_0\cup Q_1\cup \cdots\cup Q_j$. Observe that in particular, $H_0=P_1\cup L$ and $H_{k'}=G_0$.

 Informally, our proof works as follows: the set $X$ separating $P_1$ from $L$ that we construct will be the union of $X'_0$ and of at most one vertex $y_j$ for each path $Q_j$ ($j\in \sg{1,\ldots, k'}$). We will choose the vertices $y_j$ as follows: we start considering the rightmost (i.e. maximal with respect to $\leq_1$) vertex $y_1$ which is in the intersection of $L$ and of some path $Q_{j_1}$ with $j_1\in \sg{1,\ldots, k'}$. Without loss of generality, we may assume that $j_1=1$. Note that by maximality of $y_1$, the set $X'_0\cup \sg{y_1}$ separates the vertices of $L[y_1,c_L]$ from the ones of $P_1$ in $G_0$ (and thus also in $H_1$). 
 Moreover, as $L$ is $\nearrow$-free, none of the vertices from $L\cup Q_{1}$ which are smaller or equal to $y_1$ with respect to $\leq_1$ can belong to $P_1$, hence $X'_0\cup \sg{y_1}$ separates the vertices of $L\cup Q_1$ from the ones of $P_1$ in $H_1$. 
 In particular, observe that if none of the paths $Q_2,\ldots, Q_{k'}$ contains a vertex from $L\cup Q_1$ which is smaller or equal to $y_1$ with respect to $\leq_1$, then it means that $X'_0\cup \sg{y_1}$ separates $L$ from $P_1$ in $G$. We may thus assume that there exists some vertex $y_2\leq_1 y_1$ that belongs to $V(Q_{j_2})\cap V(Q_1\cup L)$ for some $j_2\geq 2$. Again, we choose such a vertex $y_2$ maximal with respect to $\leq_1$, and assume without loss of generality that $j_2=2$. By maximality of $y_2$, the set $X'_0\cup \sg{y_1, y_2}$ separates the vertices of $L$ which are greater or equal to $y_2$ from the ones of $P_1$ in $H_2$, and as $L$ is $\nearrow$-free, one can check that none of the vertices from $L\cup Q_{1}\cup Q_2$ which are smaller than $y_2$ with respect to $\leq_1$ can belong to $P_1$. We then repeat the construction inductively, by choosing at each step some vertex $y_{j+1}\leq_1 y_j$ maximal with respect to $\leq_1$ that belongs to $V(L\cup Q_1\cup \cdots \cup Q_j)\cap V(Q_{i_{j+1}})$ for some $i_{j+1}\geq j+1$ (without loss of generality, we then assume that $i_{j+1}=j+1$), and claim that at each step, the set $X'_0\cup \sg{y_1,\ldots, y_{j+1}}$ separates the vertices of $L$ from the ones of $P_1$ in $H_{j+1}$, and that none of the vertices from $L\cup Q_{1}\cup \cdots \cup Q_{j+1}$ which are smaller than $y_{j+1}$ with respect to $\leq_1$ can belong to $P_1$.
 See \Cref{fig: Induction} for an example.

 \begin{figure}[htb]
  \centering  
  \includegraphics[scale=0.6]{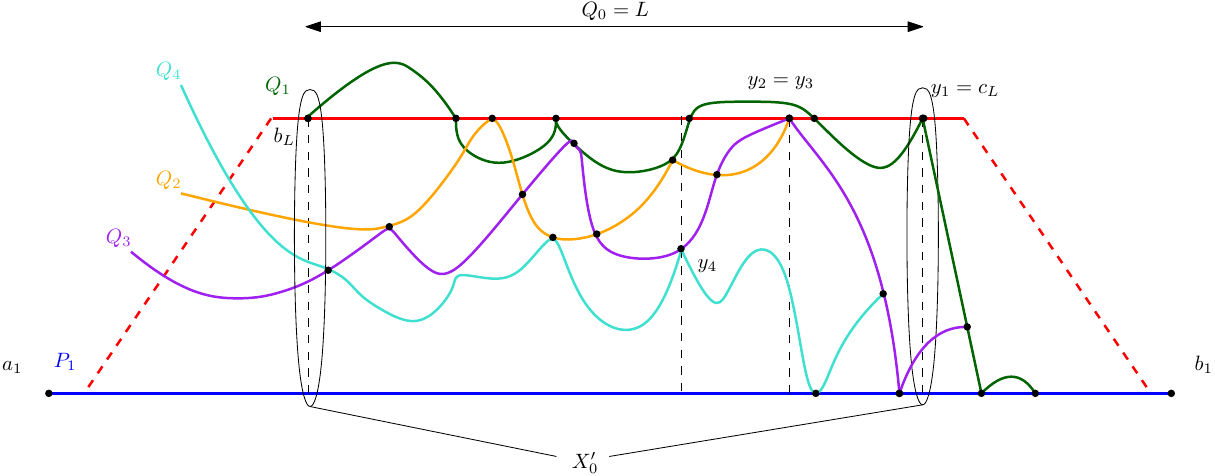}
  \caption{In the following example, as $Q_2$ and $Q_3$ both have their rightmost intersection with $W_1$ at the same vertex, we will have $y_2=y_3$, and we can choose $Q_2$ and $Q_3$ in any order in the proof of \Cref{lem: left-to-right}.}  
  \label{fig: Induction}
 \end{figure}
 
 Formally, we set $y_0:=c_L$ and $W_0:=V(L)$. 
 We construct a sequence of vertices $y_\ell\leq_1\cdots\leq_1 y_0$ and a sequence of sets $W_0, \ldots, W_\ell$ for some $0 \leq \ell \leq k'$.
 These sequences are constructed by induction as follows. 
 Let $j\geq 0$ and assume that we already constructed $y_0, \ldots, y_{j}$ and $W_0,\ldots, W_{j}$ for some $j<k'$, such that, up to permuting the indices of the paths $Q_1,\ldots,Q_{k'}$, for every $j'\leq j$, $y_{j'}\in V(Q_{j'})$. We now let $y_{j+1}$ be any vertex which is maximal with respect to $\leq_1$ such that 
 \begin{itemize}
  \item $y_{j+1}\leq_1 y_j$,
  \item $y_{j+1}\in W_j\cap V\left(\bigcup_{s\geq j+1} Q_s\right)$.
 \end{itemize}
 If no such vertex exists, then we stop the construction here and set $\ell:=j$.
 Otherwise, we let $y_{j+1}$ be such a vertex, and, up to permuting the indices of the paths $Q_{j+1},\ldots, Q_{k'}$, we assume that $y_{j+1}\in V(Q_{j+1})$. 
 We write $Q_{j+1}=Q_{j+1}^-\cdot Q_{j+1}^+$, where $Q_{j+1}^-$ (resp. $Q_{j+1}^+$) denote the subpath of $Q_{j+1}$ containing the vertices which are smaller or equal (resp. greater or equal) to $y_{j+1}$ with respect to $\leq_1$. Note that $Q_{j+1}^-$ and $Q_{j+1}^+$ are internally disjoint and share $y_{j+1}$ as an endvertex. We set $W_{j+1}:=W_j\cup  V(Q_{j+1}^-)$.
 See \Cref{fig: Induction} for an illustration.
 
  We prove by induction over $j\geq 0$ the following properties. 
 \begin{enumerate}[label=(\roman*)]
 \item\label{item: i} There exists a path $R_{j}$ from $y_{j}$ to $c_L$ in $G_0$ which is parallel to $P_1$;
 \item\label{item: ii} for every $y\in W_{j}$, if $y>_1 y_{j}$, then $y\notin V\left(\bigcup_{s\geq j}Q_s\right)$;
 \item\label{item: iii} for every $j'\leq j$, $V(Q_{j'}^+)\cap W_j\subseteq \sg{y_{j'}}$;
 \item\label{item: iv} $W_j\cap U_L=\emptyset$;
 \item\label{item: v} $X'_0\cup \sg{y_1, \ldots, y_j}$ separates $W_j$ from $U_L$ in $H_j$.
\end{enumerate}
 Note that when $j=0$, items \ref{item: i} to \ref{item: v} are immediate (note that \ref{item: iv} holds in this case because $L$ is $\nearrow$-free, and thus, $P_1$ must be internally disjoint from $L$). We now assume that $0\leq j<k'$, and that items \ref{item: i} to \ref{item: v} hold for any $j'\leq j$, and we prove that they are also satisfied for $j+1$. 
 
 We start showing \ref{item: i}. First, as $y_{j+1}\in W_j$, there exists some $0 \leq j'\leq j$ such that $y_{j+1}\in V(Q_{j'})$ and $y_{j+1}\leq_1 y_{j'}$. We consider the path $R_{j'}$ given by induction hypothesis. In particular, both $y_{j+1}$ and $y_{j'}$ belong to $Q_{j'}$, hence by \Cref{lem: parallel-concatenation}, the path $R_{j+1}:=Q_{j'}[y_{j+1}, y_{j'}]\cdot R_{j'}$ is parallel to $P_1$, and connects $y_{j+1}$ to $c_L$ in $G_0$. This proves \ref{item: i}.

 We now show \ref{item: ii}. Assume for sake of contradiction that there exists $y\in W_{j+1}$ such that $y>_1 y_{j+1}$ and $y\in V\left(\bigcup_{s\geq j+1}Q_s\right)$. 
 By definition of $W_{j+1}$, note that if $y>_1 y_j$, then $y\in W_j$. In particular, by induction hypothesis, if $y>_1 y_j$, then $y\notin V\left(\bigcup_{s\geq j}Q_s\right)$, and thus $y\notin V\left(\bigcup_{s\geq j+1}Q_s\right)$. 
 We thus have $y_{j+1}<_1 y\leq_1 y_j$. However, we now obtain a contradiction with the maximality of $y_{j+1}$ with respect to $\leq_1$, as $y\in W_j\cap V\left(\bigcup_{s\geq j+1}Q_s\right)$.

 Let us now prove \ref{item: iii}, that is, that for every $j'\leq j+1$, $V(Q_{j'}^+)\cap W_{j+1}\subseteq\sg{y_{j'}}$. First, note that if $j'\leq j$, as the vertices of $W_{j+1}\setminus W_{j}$ are exactly the internal vertices of $Q_{j+1}^-$, and as $y_{j+1}\leq_1 y_{j'}$, then the only possible vertex of $Q_{j'}^+$ belonging to $W_{j+1}\setminus W_j$ is $y_{j'}$. In particular, by induction hypothesis we immediately deduce that $V(Q_{j'}^+)\cap W_{j+1}\subseteq\sg{y_{j'}}$. It remains to check \ref{item: iii} when $j'=j+1$. Note that it follows from \ref{item: ii}, as every vertex of $Q_{j+1}^+$ which is distinct from $y_{j+1}$ is greater than $y_{j+1}$ with respect to $<_1$.

 To show \ref{item: iv}, assume for sake of contradiction that there exists a vertex $x$ in $U_L\cap W_{j+1}$. By induction hypothesis, $x\notin W_j$, hence $x$ is an internal vertex of $Q_{j+1}^-$. In particular,  $x<_1 y_{j+1}$, thus, as $Q_{j+1}$ is parallel to $P_1$, the subpath $Q_{j+1}[x, y_{j+1}]$ must be parallel to $P_1$. We let $R_{j+1}$ be the path from $y_{j+1}$ to $c_L$ given by \ref{item: i}. In particular, as $y_{j+1}\leq_1 c_L$, \Cref{lem: parallel-concatenation} implies that the path $Q_{j+1}[x, y_{j+1}]\cdot R_{j+1}$ is also parallel to $P_1$, and connects $x\in U_L$ to $c_L\in V(L)$, with $x\leq_1 c_L$, contradicting that $L$ is $\nearrow$-free.
 
 We now prove \ref{item: v}, i.e., that $X'_0\cup\sg{y_1,\ldots, y_{j+1}}$ separates $W_{j+1}$ from $U_L$ in $H_{j+1}$. Assume for sake of contradiction that this is not the case. In particular, as by \ref{item: iv}, $W_{j+1}\cap U_L=\emptyset$, there should exist an edge $yy'$ in $H_{j+1}-(X'_0\cup \sg{y_1,\ldots, y_{j+1}})$, with $y\in W_{j+1}$ and $y'\notin W_{j+1}$. By definition of $H_{j+1}$ and $W_{j+1}$, as the edge $yy'$ does not belong to $W_{j+1}$, it should be an edge of some path $Q_{j'}^+$, for some $1\leq j'\leq j+1$. In particular, \ref{item: iii} then implies that $y=y_{j'}$, a contradiction as $R$ avoids $\sg{y_1, \ldots, y_{j'}}$. 
 \smallbreak
 
 We now conclude the proof of \Cref{lem: left-to-right}. By definition of $\ell$, the sequence $y_0,\ldots,y_{\ell}$ cannot be extended further, that is, for every $\ell<s\leq k'$, $W_{\ell}$ is disjoint from $\sg{y\in V(Q_{s}): y\leq_1 y_{\ell}}$. We set 
 $X:=X'_0\cup \sg{y_1,\ldots, y_{\ell}}$. By \ref{item: v}, $X$ separates $W_{\ell}$, and thus $L$ from $U_L$ in $H_{j+1}$. It remains to show that $X$ also separates $W_{\ell}$ from $U_L$ in $G_0$. 
 Assume for sake of contradiction that it is not the case. In particular, as $Q_1,\ldots, Q_{k'}$ edge-cover $G_0$, there would 
 exist an edge $yy'\in E(Q_s)$ for some $s>\ell$, and some vertices $y\in W_{\ell}$ and $y'\in V(G_0)\setminus W_{\ell}$, such that $b_L\leq_1 y\leq_1 c_L$ and $b_L\leq_1y '\leq_1 c_L$. By our above remark, we have $y>_1y_{\ell}$. Note that we then immediately obtain a contradiction with \ref{item: ii}. This implies that $X$ indeed separates $L$ from $P_1$ in $G_0$.
 Finally, observe that 
 $$|X|\leq |X'_0|+k'\leq 3k'\leq 18k$$
\end{proof}

\begin{proof}[Proof of \Cref{lem: keyP1}]
 By \Cref{lem: QinA,lem: QinP2}, there exist two subpaths $Q_1, Q_2$ of $Q$ such that $Q=Q_1\cdot Q_2$, and such that for each $i\in \sg{1,2}$, $Q_i$ is either $\nearrow$-free or $\searrow$-free. In particular, by \Cref{lem: left-to-right}, there exist two sets $X_1, X_2$ of size at most $18k$ such that for each $i\in \sg{1,2}$, $X_i$ separates in $G_0$ the set $V(Q_i)$ from $V(P_1)$. In particular, the set $X:=X_1\cup X_2$ then separates $V(Q)$ from $V(P_1)$. Moreover, we claim that when going through the proof of \Cref{lem: left-to-right}, one gets that $X_1$ and $X_2$ both contain a common layer $V_{i}$ (for $i:=\iota(c_{Q_1})=\iota(b_{Q_2})$), allowing us to slightly reduce our upper bound on $|X_1\cup X_2|$ to obtain 
 $|X_1\cup X_2|\leq 30k$ (instead of $36k$).
\end{proof}

\section{Exact Bounds for \texorpdfstring{$\boldsymbol{k\leq 3}$}{k <= 3}}
\label{sec: 3paths}
In this section, we prove some optimal upper-bounds on the pathwidth of graphs that are edge-coverable by $k$ shortest paths, when $k\in \sg{2,3}$.

\mainthreepaths*

\subsection{Two Paths}
 As a warm-up, we start describing the structure of graphs edge-coverable by $2$ shortest paths, implying an immediate proof of \Cref{thm: 3-paths} when $k=2$.
 
\begin{definition}
\label{def: skewer}
A \emph{skewer} is a graph $G$ for which there exists a sequence of pairwise distinct vertices $x_1, \ldots, x_{\ell}$ ($\ell\geq 1$) such that: 
\begin{itemize}
 \item for each $i\in \sg{1, \ldots, \ell-1}$, the vertex $x_i$ is connected to the vertex $x_{i+1}$ by two internally disjoint paths $Q_i, R_i$ having the same length;
 \item for each $i\in \sg{1, \ldots, \ell-1}$, the paths $Q_i, R_i$ are separated from $G-(Q_i\cup R_i)$ by the set $\sg{x_i, x_{i+1}}$;
 \item  there exist four internally disjoint paths  $Q_0, R_0, Q_{\ell}, R_{\ell}$ (which do not necessarily have the same length), such that $Q_0, R_0$ (resp.\ $Q_{\ell},R_{\ell}$) are separated from $G-(Q_0\cup R_0)$ (resp.\ $G-(Q_\ell\cup R_{\ell})$) by the set $\sg{x_1}$ (resp.\ $\sg{x_\ell}$).
\end{itemize}
\end{definition}
 See \Cref{fig: skewer} for an illustration of a skewer.
 For each $i\in \sg{0, \ldots, \ell}$, the graph $C_i:=Q_i\cup R_i$ is called the \emph{piece} of $x_i$. Note that we allow the paths $Q_i$ and $R_i$ to have length $1$, in which case the graph $C_i$ is just the edge $x_ix_{i+1}$ (recall that the graphs we consider are simple), while if $Q_i$ and $R_i$ have length at least $2$, then $C_i$ is an even cycle.  
 Note that every skewer $G$ has a $2$-layering, and thus pathwidth at most $2$, and that moreover, for every $i\in \sg{0,\ldots, \ell}$, and every pair of vertices $a,b$ with $a\in V(Q_i)$ and $b\in V(R_i)$, there exists a path decomposition of $G$ of width $2$ in which there exists a bag equals to $\sg{a,b}$.

 \begin{figure}[htb]
    \centering  
    \includegraphics[scale=1]{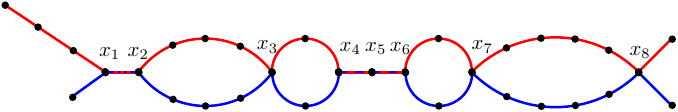}
    \caption{A skewer. The blue paths represent the paths $R_i$ and the red paths represent the paths $Q_i$.}
    \label{fig: skewer}
\end{figure}

\begin{lemma}
\label{clm: ADN}
 Let $P,Q$ be two shortest paths in a graph $G$. Then there cannot exist three vertices $a,b,c\in V(P)\cap V(Q)$ such that $b$ is an internal vertex of $P[a,c]$ and such that $c$ is an internal vertex of $Q[a,b]$.
\end{lemma}

\begin{proof}
    Assume for the sake of contradiction that there exist three such vertices $a, b,c \in V(P) \cap V(Q)$. 
    Since $b$ is an internal vertex of $P[a,c]$, $\dist(a,b) < \dist(a,c)$. 
    But since $c$ is an internal vertex of $Q[a,b]$, $\dist(a,c) < \dist(a,b)$, a contradiction.
\end{proof}

\begin{proposition}
 \label{prop: skewer}
 Every connected graph $G$ which is edge-coverable by two isometric paths is a skewer. In particular, $\pw(G)\leq 2$. 
\end{proposition}

\begin{proof}
 We let $P_1, P_2$ be two isometric paths that edge-cover $G$, and let $x_1, \ldots, x_{\ell}$ denote the vertices from $V(P_1)\cap V(P_2)$, such that for each $i<j$, $x_i$ is before $x_j$ on $P_1$. By \Cref{clm: ADN}, for each $i<j$, $x_i$ is also before $x_j$ on $P_2$. It thus implies that each $x_i$ separates in $G$ all vertices appearing before it in $P_1$ and $P_2$ from all vertices appearing after it in $P_1$ and $P_2$. In particular, it easily follows that $G$ is a skewer.
\end{proof}

\subsection{Three Paths}
This subsection consists in a proof of \Cref{thm: 3-paths} when $k=3$, which we restate here for convenience.

\begin{theorem}
    Let $G$ be a graph which is edge-coverable by $3$ isometric paths. Then $\pwr{G} \leq 3$. 
\end{theorem}

In the remainder of the subsection, we let $G$ denote a graph which is edge-coverable by three shortest paths $P_1, P_2, P_3$. We also assume that for each $i\in \sg{1,2,3}$, $P_i$ is a $a_ib_i$-path for some vertices $a_i,b_i\in V(G)$. Moreover, note that if the $P_i$'s do not all belong to the same connected component of $G$, then every connected component of $G$ is coverable by at most $2$ shortest paths, hence by \Cref{prop: skewer}, $\pw(G)\leq 2$. We thus moreover assume that $G$ is connected.

For two sequences $\cP_1 = (X_1, \dots, X_k)$ and $\cP_2 = (Y_1, \dots, Y_\ell)$ of vertex subsets such that $X_k = Z = Y_\ell$, we say that $\cP = (X_1, \dots, X_k, Z, Y_1, \dots Y_\ell)$ is \emph{the glue of}$\cP_1$ and $\cP_2$ \emph{by the bag} $Z$. The following three lemmas deal with some first easy cases.

\begin{lemma}\label{cl: 3-paths intersects}
    If the intersection of two of the paths $P_1, P_2, P_3$ is included in the third one, then $\pw(G) \leq 3$.
\end{lemma}

\begin{proof}
 Assume without loss of generality that $V(P_1)\cap V(P_2)\subseteq V(P_3)$, and that $G$ is connected. 
 Note that if one of the two paths $P_1, P_2$ is disjoint from $P_3$, then $G$ is not connected, and we get a contradiction, hence both paths intersect $P_3$. For each $i\in \sg{1,2}$, we let $u_i, v_i$ denote the two vertices from $V(P_i)\cap V(P_3)$ maximizing $d(u_i, v_i)$, and such that $d(a_3,u_i)\leq d(a_3, v_i)$, and we let $B_i:=P_i[u_i, b_i]$ (see \Cref{fig: 3paths_intersection}). Then the graph $H:=P_3\cup B_1\cup B_2$ admits a $3$-layering $(V_0, \ldots, V_r)$, where $r:=\ecc_H(a_3)$, and for each $i\in \sg{0, \ldots, r}$, $V_i:=\sg{u\in V(H): d(a_3, u)=i}$; this follows from the observation that the three paths $P_3, B_1, B_2$ are vertical with respect to $a_3$ (note that the hypothesis that $V(P_3)\subseteq V(P_1)\cap V(P_2)$ is used here). In particular, note that this layering satisfies the property that each path $P_j$ intersects at most once each layer $V_i$. 
 Observe now that $G$ is obtained from $H$ by attaching the two pendant paths $A_1:=P_1[a_1, u_1]$ and $A_2:=P_2[a_2, u_2]$ to the vertices $u_1$ and $u_2$. We claim that it is then not hard to construct a $3$-layering of $G$ obtained from the one of $H$ by adding these two pendant paths to the layering $(V_i)_{0\leq i\leq r}$.
\end{proof}

\begin{figure}[htb]
    \centering  
    \includegraphics[scale=1]{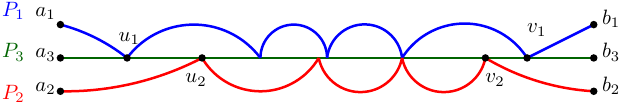}
    \caption{Proof of \Cref{cl: 3-paths intersects}.}
    \label{fig: 3paths_intersection}
\end{figure}

\begin{lemma}\label{cl:3-paths 2 intersections}
    If one of the three paths $P_1,P_2,P_3$ intersects the other two at most twice in total, then $\pw(G) \leq 3$.
\end{lemma}

\begin{proof}
    Assume without loss of generality that $|V(P_3)\cap (V(P_1)\cup V(P_2))|\leq 2$. 
    Note that if $P_3$ does not intersect $V(P_1)\setminus V(P_2)$ (or symmetrically, $V(P_2)\setminus V(P_1)$), then we have $V(P_3)\cap V(P_1)\subseteq V(P_2)$ (and symmetrically, $V(P_3)\cap V(P_2)\subseteq V(P_1)$), in which case \Cref{cl: 3-paths intersects} implies that $\pw(G)\leq 3$. We may thus assume that $P_3$ intersects $V(P_1)\cup V(P_2)$ in exactly two vertices $a\in V(P_1)\setminus V(P_2)$ and $b\in V(P_2)\setminus V(P_1)$.  
    
    We now set $H:=P_1\cup P_2$. By \Cref{prop: skewer}, $H$ is a skewer, and we let $x_1,\ldots, x_{\ell}$, $R_0, \ldots, R_{\ell}$, $Q_0, \ldots, Q_{\ell}$, denote its corresponding components with the notations of \Cref{def: skewer}. 
    Assume first that there exists some $i\in \sg{0,\ell}$, such that $a\in R_i$ and $b\in Q_i$. Then $H$ has a path decomposition of width $2$ that contains a bag equal to $\sg{a,b}$. In particular, as $\sg{a,b}$ separates $V(P_3)$ from $V(H)$ in $G$, we can insert after the bag $\sg{a,b}$ of this path decomposition some path decomposition of width $1$ of $P_3$ to which we add in each bag the set $\sg{a,b}$. This gives a path decomposition of width at most $3$ of $G$. 
    
    Assume now that we are not in the previous situation. In particular, as we assumed that $a\in V(P_1)\setminus V(P_2)$ and $b\in V(P_2)\setminus V(P_1)$, up to symmetry, there exist $i\neq j$ such that $a\in R_i$ and 
    $b\in Q_j$. We now let $P'_1:=R_0\cdots R_{j-1}\cdot Q_j \cdot R_{j+1}\cdots R_{\ell}$ and $P'_2:=Q_0\cdots Q_{j-1}\cdot R_j \cdot Q_{j+1}\cdots Q_{\ell}$. Observe that $H$ is edge-covered by $P'_1, P'_2$, and that $P'_1, P'_2$ are shortest paths in $G$, such that $a,b\in V(P'_1)$. In particular, we can again conclude as above, using \Cref{cl: 3-paths intersects} with respect to the paths $P'_1,P'_2,P_3$ that edge-cover $G$.
\end{proof}

\begin{lemma}\label{lem:two common points}
    If there are at least two vertices that simultaneously belong to the three paths $P_1, P_2, P_3$, then $\pw(G) \leq 3$
\end{lemma}

\begin{proof}
    \label{proof:two-common-points}
    Let $u$ and $v$ be two distinct vertices of $V(P_1) \cap V(P_2) \cap V(P_3)$.
    Without loss of generality, we assume that $a_i, u, v, b_i$ appear in this order on $P_i$ for $i=1,2,3$. 
    For $i=1,2,3$, we denote $A_i$ the $a_iu$-subpath of $P_i$, and $B_i$ the $ub_i$-subpath of $P_i$.
    Then, the graph $G$ admits a $3$-layering $(U_\ell, \dots, U_1, V_0, \dots,V_r)$ such that $U_i := \{x \in V(A_1) \cup V(A_2) \cup V(A_3) : \dist(u, x) = i\}$ and $V_i := \{x \in V(B_1) \cup V(B_2) \cup V(B_3): \dist_H(u, x) = i\}$; this follows from the fact that $\sg{u}$ separates the vertices on the $A_i$'s from the vertices on the $B_i$'s in $G$, and that the $A_i$'s and $B_i$'s are all vertical w.r.t. $u$. 
\end{proof}

We say that a path $P$ \emph{bounces} between two paths $Q_1, Q_2$ if there exists $i\in \sg{1,2}$ such that $P$ has a subpath $P'$ with both endvertices on $Q_i$, and such that one of the internal vertices of $P'$ intersects $Q_{3-i}$.

By \Cref{cl:3-paths 2 intersections}, we can assume from now on that $P_1$ intersects the other paths at least twice, and by \Cref{lem:two common points}, that at most one of the vertices in these intersections lies simultaneously on the three paths $P_1, P_2, P_3$. 
In particular, up to exchanging $a_1$ and $b_1$, and $P_2$ and $P_3$, we may assume that if $u_0$ denotes the vertex on $P_1$ which is the closest to $a_1$, and belongs to one of the two paths $P_2, P_3$, then $u_0\in V(P_1) \cap V(P_2) \setminus V(P_3)$. 
We set $A_i := P_i[a_i,u_0]$ and $B_i:=P_i[u_0,b_i]$ for $i= 1,2$.

\begin{remark}
 \label{rem: A1}
 Observe that our choice of $u_0$ implies that $V(A_1)\cap (V(P_2)\cup V(P_3))=\sg{u_0}$.
\end{remark}
By \Cref{cl: 3-paths intersects} and \Cref{cl:3-paths 2 intersections}, we can assume that $P_3$ intersects $B_1-u_0$, and at least one of the paths between $A_2-u_0$ and $B_2-u_0$, say $B_2-u_0$ by symmetry. 
In particular, at least one of the following  situations occurs (up to symmetrically exchange the roles of $A_2$ and $B_2$)
\begin{enumerate}[label=(\arabic*)]
    \item\label{it: 1} $P_3$ bounces between two of the paths $A_2, B_2$ and $B_1$,
    \item\label{it: 2} $P_3$ does not intersect $A_2$,
    \item\label{it: 3} $P_3$ intersects the three paths $A_2, B_2, B_1$ in this order,
    \item\label{it: 4} $P_3$ intersects the three paths $A_2, B_1, B_2$ in this order.
\end{enumerate}
To conclude the proof of \Cref{thm: 3-paths}, we will prove that $\pw(G) \leq 3$ in each of these cases in the following separate lemmas. This is the most technical part of the proof, especially for cases \ref{it: 1} and \ref{it: 2}.
The next four lemmas respectively deal with each of the four cases. We start with \ref{it: 1}.

\begin{lemma}\label{lem:alternatively intersect}
    If $P_3$ bounces between two of the paths in $\sg{A_2, B_1, B_2}$, then $\pw(G) \leq 3$.
\end{lemma}

\begin{proof}
     We let $u_0\in V(P_1)\cap V(P_2) \setminus V(P_3)$ be such that $P_3$ bounces between two paths in $\sg{A_2, B_1, B_2}$. 
    Note that by Claim \ref{clm: ADN}, $P_3$  cannot bounce between $A_2$ and $B_2$. 
    We assume that $P_3$ bounces between $B_1$ and $B_2$, and claim that the case where $P_3$ bounces between $B_1$ and $A_2$ is symmetric.
    Let $B_3$ be the longest subpath of $P_3$ with both extremities $u_1$ and $u_3$ in $V(B_1) \cup V(B_2)$. 
    Since $P_3$ bounces between $B_1$ and $B_2$, $B_3$ has length at least $2$.
    We moreover assume without loss of generality that $0 < \dist(u_0, u_1) \leq \dist(u_0, u_3)$ (recall that $u_0\notin V(P_3)$). 

    We claim that $B_3$ is vertical with respect to $u_0$. This is immediate if $B_3$ has both endvertices on $B_1$, or symmetrically on $B_2$, as both $B_1, B_2$ are shortest paths in $G$ starting from $u_0$. Assume now without loss of generality that $u_1\in V(B_1)$, while $u_3\in V(B_2)$. Then one can easily see that, as $B_3$ bounces between $B_1$ and $B_2$, there exist two distinct vertices $u_2\in V(B_2), u'_2\in V(B_1)$ such that $u_1, u_2, u'_2, u_3$ are pairwise distinct and appear in this order on $B_3$. In particular, in the latter case, $B_3[u_1, u'_2]$ and $B_3[u_2, u_3]$ both bounce between $B_1$ and $B_2$, and have both endvertices respectively on $B_1$, and on $B_2$, so they are both vertical with respect to $u_0$. As they edge-cover $B_3$, and share the non-trivial proper subpath $B_3[u_2,u'_2]$, it follows that $B_3$ is also vertical with respect to $u_0$.

    We now prove that $B_3$ does not intersect $A_2$. 
    Assume for sake of contradiction that there exists a vertex $w$ in $V(B_3)\cap V(A_2)$.
    As $u_0\notin V(P_3)$, $w \neq u_0$.
    As $B_3$ bounces between $B_1$ and $B_2$, note that $V(B_3) \cap V(B_2)$ contains a vertex $v$ distinct from $u_0$. 
    By the previous paragraph, observe that the path $Q:=B_3[v,w]$ is vertical with respect to $u_0$, hence $d(v,w)<\max(d(u_0,v), d(u_0,w))$.
    In particular, we then obtain a contradiction as $P_2[v,w]$ has length $d(u_0,v)+d(u_0,w)$, and should be a shortest $vw$-path in $G$.

    We now consider the subpaths $A_3,C_3$ of $P_3$ such that $P_3 := A_3 \cdot B_3 \cdot C_3$. 
    By construction of $B_3$, both $A_3-u_1$ and $C_3-u_3$ are disjoint from $V(B_1)\cup V(B_2)$.

    We prove that $C_3$ does not intersect $A_2$.
    Assume for sake of contradiction that $C_3$ intersects $A_2$, on a vertex $w$ (see \Cref{fig: C2}).
    First, note that $B_3$ should then intersect $B_2$ in at most one vertex, as otherwise, the paths $P_2,P_3$ would contradict \Cref{clm: ADN}. 
    We then assume that $B_3$ intersects $B_2$ in exactly one vertex $u_2$. Since $B_3$ bounces between $B_1$ and $B_2$, in particular $u_2$ is an internal vertex of $B_3$, and we must have $u_1, u_3\in V(B_1)\setminus V(B_2)$.
    We now set $a:=d(u_1,u_2)$, $b:=d(u_2,u_3), c:=d(u_3,w), d:=d(u_0,w), e:=d(u_0,u_2)$ and $f:=d(u_0,u_1)$, so that the situation is as depicted in \Cref{fig: C2}. 
    Since $B_3$ is vertical with respect to $u_0$, we have $e = a + f$. 
    Moreover, since both paths $P_2[w, u_2]$ and $P_3[u_2,w]$ are shortest paths, we have $\dist(u_2, w) = b+c = e+d$.
    Moreover, as $P_3[u_1,w]$ is a shortest path, its length is shorter than the one of $P_2[u_0,w]\cdot P_1[u_0,u_1]$, giving $a+b+c\leq f+d$. Combining all previous inequalities then gives $a+e\leq f$, and thus $a=0$, contradicting that $u_2$ is an internal vertex of $B_3$.
    We then obtain, as desired, that $C_3$ does not intersect $A_2$, and thus that $C_3$ only intersects $P_1,P_2$ in $u_3$.

    \begin{figure}
    \centering  
    \includegraphics[scale=1, page=2]{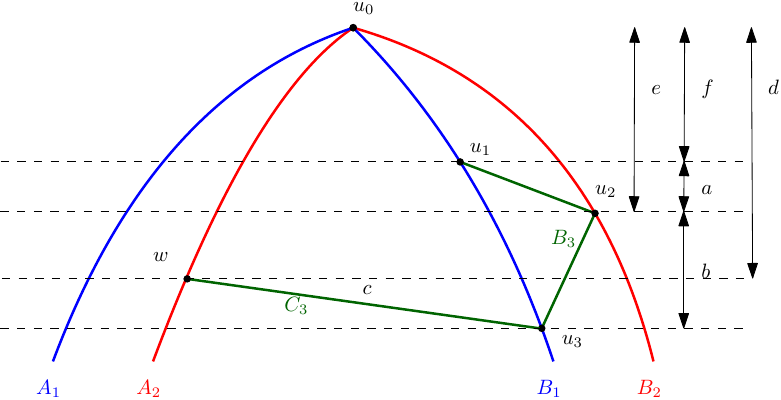}
    \caption{Configuration from the proof of \Cref{lem:alternatively intersect}. The dotted lines represent the layers in the BFS from $u_0$.}
    \label{fig: C2}
    \end{figure}

    We now prove that $V(B_1) \cap V(A_2) = \{ u_0 \}$. 
    Assume for sake of contradiction that there exists a vertex $v\in V(B_1)\cap V(A_2)$ distinct from $u_0$. As $B_3$ bounces between $B_1$ and $B_2$ and as $u_0\notin V(P_3)$, there exist two distinct vertices $x,y$ of $B_3$ such that $0< d(x, u_0)<d(y,u_0)$, and such that $x\in V(B_1)$ and $y\in V(B_2)$. We distinguish two cases (depicted in \Cref{fig: bounce_special}), according whether $d(u_0, x)\leq d(u_0, v)$ or not. Note that in both cases, as $A_2, B_1, B_2$ and $B_3$ are vertical with respect to $u_0$, and as $x, v\neq u_0$, we obtain a contradiction, as the path $B_1[v,x]\cdot B_3[x,y]$ is always strictly shorter than the path $P_2[v, y]$, which is supposed to be a shortest path in $G$.
    
    \begin{figure}[htb]
    \centering  
    \includegraphics[scale=0.65]{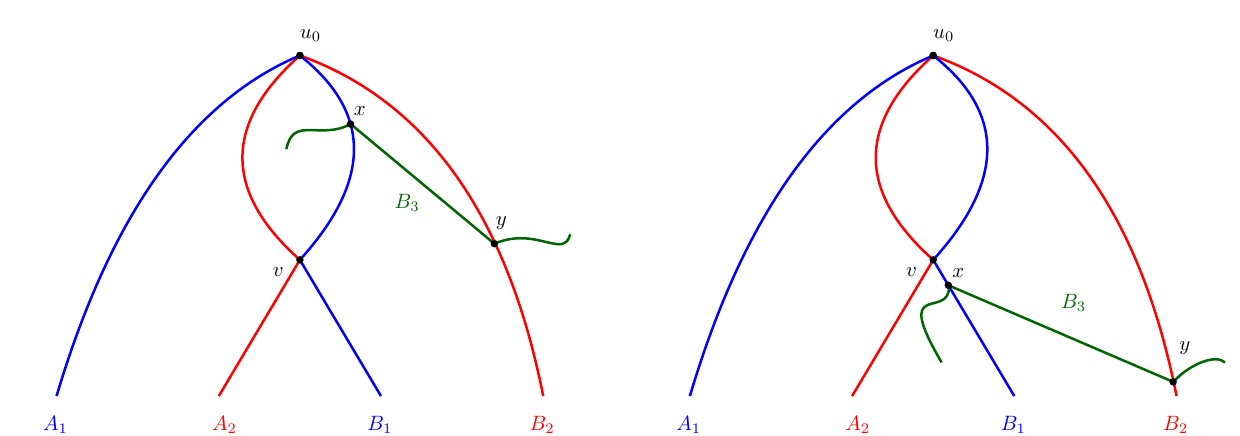}
    \caption{The two different cases considered in the proof of \Cref{lem:alternatively intersect}, when we assume that $A_2$ and $B_2$ intersect each other. On the left, the situation where $d(u_0, x)\leq d(u_0, v)$, and on the right, the situation where $d(u_0, x)\geq d(u_0, v)$.}
    \label{fig: bounce_special}
    \end{figure}
    
    All previous paragraphs together with \Cref{rem: A1} thus imply that we can write $G=G_1\cup G_2$, where $G_1:=A_1\cup A_2\cup A_3$ and $G_2:=B_1\cup B_2\cup B_3\cup C_3$, and $\sg{u_0, u_1}$ separates $G_1$ from $G_2$.
    We now show that $G_1$ (resp. $G_2$) admits a path decomposition of width at most $3$ whose rightmost (resp. leftmost) bag contains $\sg{u_0, u_1}$. This will immediately imply that $\pw(G)\leq 3$.
    
    We first construct such a path decomposition $\cP_2$ for $G_2$. Observe that $B_1, B_2$ are both vertical with respect to $u_0$. Moreover, recall that $C_3-u_3$ is disjoint from $V(B_1) \cup V(B_2)$, hence,
    as $B_3$ is vertical with respect to $u_0$, it implies that $B_3\cdot C_3$ is also vertical with respect to $u_0$. It thus implies that the partition $(V_i)_{i\geq 0}$ of $V(G_2)$ defined by setting for each $i\in \sg{0,\ldots, r}$ (where $r:=\ecc_{G_2}(u_0)$), $V_i:=\sg{x\in V(G_2): d(u_0, x)=i}$ is a $3$-layering of $G_2$. Moreover, note that if we set $s:=d(u_0,u_1)$, then the two subsequences $\mathcal S_1:=(V_0, \ldots, V_{s})$ and $\mathcal S_2:=(V_s, \ldots, V_{r})$ of $(V_i)_{0\leq i\leq r}$ are respectively $2$- and $3$-layerings of the subgraphs $G'_2, G''_2$ of $G_2$ induced respectively by $V_0\cup \cdots\cup V_s$ and by $V_s\cup \cdots \cup V_r$ (see \Cref{fig: bounce-layer}). In particular, two applications of \Cref{lem: k-layered} give path decompositions $\cP'_2, \cP''_2$ respectively of $G'_2$ and $G''_2$ such that: 
    \begin{itemize}
     \item $\cP'_2$ has width at most $2$, and its extremities are $V_0=\sg{u_0}$ and $V_s$,
     \item $\cP''_2$ has width at most $3$, and its leftmost bag is $V_s$ (and thus contains $u_1$).
    \end{itemize}
    To obtain a path decomposition $\cP_2$ of $G_2$ with the desired properties, we can now simply add the vertex $u_1$ in every bag of $\cP'_2$, and glue the obtained path decomposition with $\cP''_2$.  
    
   \begin{figure}[htb]
    \centering  
    \includegraphics[scale=0.8]{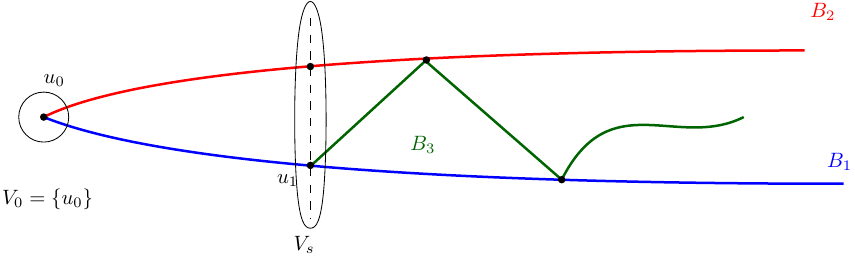}
    \caption{The $3$-layering of $G_2$ described at the end of the proof of \Cref{lem:alternatively intersect}.}
    \label{fig: bounce-layer}
    \end{figure}    

    It remains to construct a path decomposition $\cP_1$ of width $3$ of $G_1$ whose rightmost bag contains $\sg{u_0, u_1}$. 
    By \Cref{rem: A1}, $\sg{u_0}$ separates the vertices of $A_1$ from the other vertices of $G_1$.
    Moreover, by \Cref{prop: skewer}, $A_2 \cup A_3$ is a skewer. In particular, it admits a path decomposition $\cP'_1$ of width at most $2$ whose rightmost bag is $\sg{u_1}$. As $G_1$ is obtained from $A_2\cup A_3$ after attaching some pendant path $A_1$ on $u_0$, the existence of $\cP_1$ easily follows.
\end{proof}

The next lemma deals with \ref{it: 2}.

\begin{lemma}\label{lem:not intersect both}
b    If $P_3$ and $A_2$ are disjoint, then $\pw(G) \le 3$. 
\end{lemma}

\begin{proof}
    \label{proof: not intersect both}
    Suppose that $P_3$ and $A_2$ are disjoint.
    Since by \Cref{rem: A1}, $P_3$ does not intersect $A_1$, by \Cref{cl: 3-paths intersects} we may assume that $P_3$ intersects both $B_1$ and $B_2$.
    We let $v_1$ and $w_1$ (resp. $v_2$ and $w_2$) be the vertices of $V(B_1)\cap V(P_3)$ (resp. $V(B_2) \cap V(P_3)$) such that $v_1$ (resp. $v_2$) is the closest to $u_0$ and $w_1$ (resp. $w_2$) the farthest.
    Note that it is possible that $v_1 =w_1$, or $v_2=w_2$.
    By \Cref{lem:alternatively intersect}, we may assume that $P_3$ does not bounce between $B_1$ and $B_2$, hence we may assume (up to symmetry) that $P_3$ intersects $B_1$, and then $B_2$.
    By \Cref{lem:two common points}, we may moreover assume that there is at most one vertex in $V(P_1) \cap V(P_2) \cap V(P_3)$.
    We write $P_3 = A_3 \cdot B_3 \cdot C_3$, such that $B_3$ is the subpath of $P_3$ from the last vertex of $V(P_1)\cap V(P_3)$ to the first vertex of $V(P_2)\cap V(P_3)$ (note that $B_3$ might contain only one vertex if $V(P_1) \cap V(P_2) \cap V(P_3)$ is not empty). 
    By \Cref{clm: ADN}, $B_3$ has one endpoint in $\sg{v_1, w_1}$, and the other in $\sg{v_2, w_2}$.
    Then, let $C_1 := P_1[a_1, v_1], D_1 := P_1[v_1,w_1], E_1 := P_1[w_1,b_1]$ and let $C_2 := P_2[a_2,v_2], D_2 := P_2[v_2,w_2], E_2 := P_2[w_2,b_2]$ (see \Cref{fig: cas2}). 
    
    \begin{figure}[htb]
    \centering  
    \includegraphics[scale=0.75]{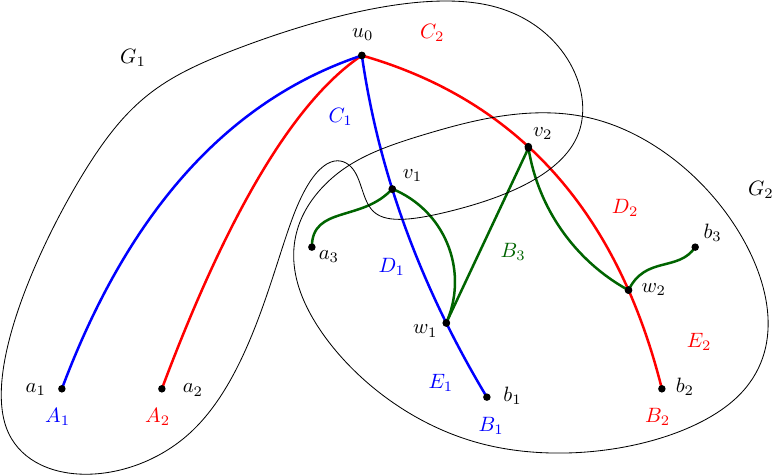}
    \caption{Configuration in the proof of \Cref{lem:not intersect both}. In this special case, the endvertices of $B_3$ are $v_2, w_1$.} 
    \label{fig: cas2}
    \end{figure}

    In what follows, we let $c_1$ be the first vertex of $P_3$ that belongs to $V(P_1)\cup V(P_2)$, and $c_2$ be the last vertex of $P_3$ that belongs to $V(P_1)\cup V(P_2)$. As we assumed that $P_3$ intersects $B_1$ and then $B_2$, we have $c_1\in V(P_1)$ and $c_2\in V(P_2)$. Moreover, note that if  
    $c_1$ is in $V(P_2)$, then we have $V(P_3)\cap V(P_1)\subseteq V(P_2)$, and we conclude that $\pw(G)\leq 3$ by \Cref{cl: 3-paths intersects}. We thus may assume that $c_1\in V(P_1)\cup V(P_3)\setminus V(P_2)$, and by symmetry that $c_2\in V(P_2)\cup V(P_3)\setminus V(P_1)$.

   \begin{claim}
    \label{clm: cas2_D}
    If one of the two inclusions $V(D_1)\cap V(B_2)\subseteq \sg{v_1, w_1}$ or $V(D_2)\cap V(B_1)\subseteq \sg{v_2, w_2}$ does not hold, then $\pw(G)\leq 3$.
   \end{claim}
   \begin{proof}[Proof of the claim]
    We assume that the first inclusion $V(D_1)\cap V(B_2)\subseteq \sg{v_1, w_1}$ does not hold, and claim that the proof if the second inclusion does not hold is symmetric. Then there exists $v\in V(D_1)\cap V(B_2)$ such that 
    $\dist(u_0,v_1) < \dist(u_0, v) < \dist(u_0,w_1)$.
    
    We now claim that we are in the same configuration as in Case \ref{it: 1}, with $c_2$ playing the role of $u_0$, where $P_1$ bounces between the paths $P_2[c_2, a_2]$ and $P_3[c_2, a_3]$, as it either intersects $v_1, v, w_1$ in this order, or $w_1, v, v_1$ in this order (see \Cref{fig: cas2_Claims}, left). We thus conclude that $\pw(G)\leq 3$ thanks to \Cref{lem:alternatively intersect}.
   \end{proof}

   \begin{claim}
    \label{clm: cas2_CE}
    If $V(E_1)\cap V(C_2)\neq \emptyset$, or $V(E_2)\cap V(C_1)\neq \emptyset$, then $\pw(G)\leq 3$.
   \end{claim}
   \begin{proof}[Proof of the claim]
    Assume that there exists a vertex $v$ in $V(E_1)\cap V(C_2)$. We claim that the case where $V(E_2)\cap V(C_1)\neq \emptyset$ is symmetric. The situation is then as depicted in the right of \Cref{fig: cas2_Claims}. We now set $P'_1:=A_1\cdot P_2[u_0, v]\cdot P_1[v,b_1]$ and $P'_2:=A_2\cdot P_1[u_0, v]\cdot P_2[v,b_2]$. Observe that $P'_1, P'_2$ are also shortest paths in $G$, and that $P'_1, P'_2, P_3$ cover the edges of $G$. In particular, note that $V(P_3)\cap V(P'_1)\subseteq V(P'_2)$, thus \Cref{cl: 3-paths intersects} applies and gives $\pw(G)\leq 3$. 
   \end{proof}

   \begin{figure}[htb]
    \centering  
    \includegraphics[scale=0.75]{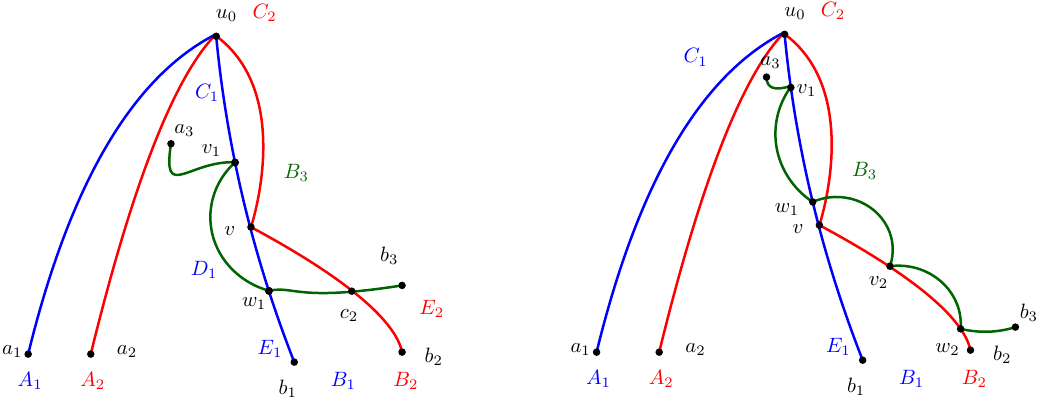}
    \caption{Left: configuration of the proof of \Cref{clm: cas2_D}. In this figure, we depicted the special case where $c_1=v_1$, however note that there is also a symmetric case where $c_1=w_1$, in which case $c_2$ is before $v$ on $P_2$.
    Right: configuration of the proof of \Cref{clm: cas2_CE}.}
    \label{fig: cas2_Claims}
    \end{figure}
    
    Now, observe that Claims \ref{clm: cas2_D} and \ref{clm: cas2_CE}, together with the definition of $v_1, v_2$ imply that we may assume that the set $\sg{v_1, v_2}$ separates in $G$ the subgraph $G_1:=A_1\cdot C_1\cup A_2\cdot C_2$ from the subgraph $G_2:=D_1\cup E_1 \cup D_2\cup E_2\cup P_3$ (see \Cref{fig: cas2}).
    By \Cref{prop: skewer}, $G_1$ is a skewer, and admits a path decomposition $\cP'_1$ of width at most $2$ such that $v_1$ is in the rightmost bag of $\cP'_1$. By adding $v_2$ to every bag of $\cP_1'$, we obtain a path decomposition $\cP_1$ of $G_1$ of width at most $3$ such that $v_1,v_2$ is the rightmost bag of $\cP_1$. 
    It thus remains to show that there exists a path decomposition of width $3$ of $G_2$ whose leftmost bag is $\sg{v_1, v_2}$.

    We now have four distinct cases, two of them being symmetric, depicted on \Cref{fig: cas2_3finals} according to the order on which the vertices $v_1, w_1, v_2, w_2$ appear on $P_3$. 
    
   \begin{figure}[htb]
    \centering  
    \includegraphics[scale=0.7]{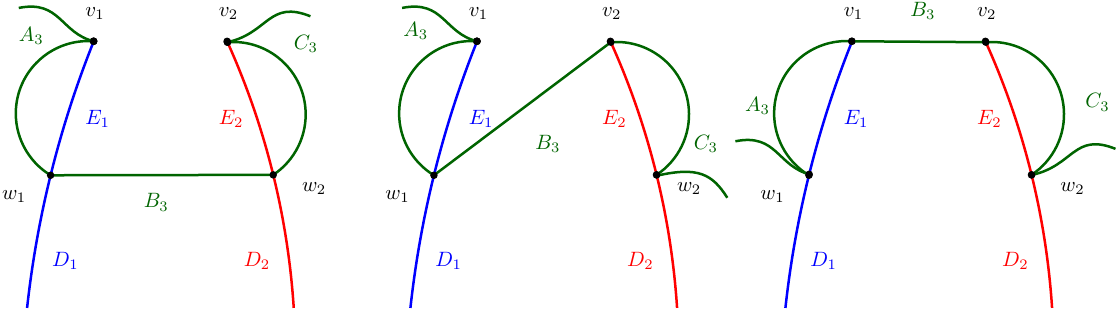}
    \caption{Possible ways for the path $P_3$ (in dark green) to intersect $P_1$ (in blue) and $P_2$ (in red) in $G_2$. Left: the case where $v_1, w_1, w_2, v_2$ appear in this order on $P_3$. Middle: the case where $v_1, w_1, v_2, w_2$ appear in this order on $P_3$. The case where $w_1, v_1, w_2, v_2$ appear in this order on $P_3$ is symmetric.
    Right: the case where $w_1, v_1, v_2, w_2$ appear in this order on $P_3$.}
    \label{fig: cas2_3finals}
    \end{figure}
    
    Note that by Claims \ref{clm: cas2_D} and \ref{clm: cas2_CE}, the only possible common vertices between $P_1$ and $P_2$ in $G_2$ are in $V(D_1)\cap V(D_2)$, and possibly, the vertices $v_1, v_2$ in the special case when $v_1=v_2$. In particular, by definition of $w_1, w_2$, the set $\sg{w_1, w_2}$ separates in $G_2$ the vertices of the graph $G_3:=D_1\cup D_2$ from the remaining vertices. By \Cref{prop: skewer}, and as $u_0\in V(P_1)\cap V(P_2)$, $G_3$ admits a path decomposition of width at most $2$ whose leftmost bag is $\sg{w_1, w_2}$. To conclude our proof, it thus only remains to show in each case that there exists a path decomposition of $G_4:=P_3\cup E_1\cup E_2$ of width at most $3$, whose leftmost bag is $\sg{v_1, v_2}$, and whose rightmost bag is $\sg{w_1, w_2}$.
    We claim that it is then not hard to construct separately in each case such a decomposition. 
\end{proof}

The next two lemmas respectively tackle cases \ref{it: 3} and \ref{it: 4}.

\begin{lemma}\label{lem:intersection semi alternatively}
    If $P_3$ intersects the subpaths $A_2, B_1, B_2$ in this order, then $\pw(G) \le 3$. 
\end{lemma}

\begin{proof}
    As $u_0\notin V(P_3)$, note that $P_3$ intersects the subpaths $A_2-u_0, B_1-u_0, B_2-u_0$.
    By \Cref{clm: ADN}, $B_1$ cannot intersect both $A_2 - u_0$ and $B_2 - u_0$. 
    By symmetry, assume that $B_1$ does not intersect $A_2 - u_0$. 
    Let $w$ be the vertex of $V(A_2) \cap V(P_3)$ which is the closest from $a_3$, and let $A_3:=P_3[a_3, w]$. As $B_1$ does not intersect $A_2\setminus w$, and as $u_0\notin V(P_3)$, we have $w\in V(P_3)\cap V(A_2)\setminus V(B_1)$. Now, observe that if $A_3$ intersects the path $P_1$, then as $P_3$ is disjoint from $A_1$, it implies that $P_3$ bounces between $A_2$ and $B_1$, in which case \Cref{lem:alternatively intersect} applies and gives $\pw(G)\leq 3$. Assume now that $A_3$ and $P_1$ are disjoint. Then, we claim that we are in the configuration of Case \ref{it: 2}, with $w$ playing the role of $u_0$, $P_1$ playing the role of $P_3$, and $P_3$ playing the role of $P_1$. In particular, 
    \Cref{lem:not intersect both} then applies and gives $\pw(G)\leq 3$.
\end{proof}

\begin{lemma}\label{lem:intersect twice then once}
    If $P_3$ intersects the subpaths $A_2, B_2, B_1$ in this order, then $\pw(G) \le 3$. 
\end{lemma}

\begin{proof}
As $u_0\notin V(P_3)$, note that $P_3$ also intersects the subpaths $A_2-u_0, B_2-u_0, B_1-u_0$. 
We let $x,y,z$ be three distinct vertices of $P_3$ such that $x\in V(A_2), y\in V(B_2)$ and $z\in V(A_1)$, and such that $y$ is an internal vertex of $P_3[x,z]$. 
    
Note that by \Cref{clm: ADN}, no vertex of $P_3[y, b_3]$ belongs to $A_2$, and no vertex of $P_3[a_3, x]$ belongs to $B_2$. Moreover, by \Cref{lem:alternatively intersect}, no vertex on $P_3[z,b_3]$ belongs to $P_2-z$. In particular, it implies that we may assume that $x$ is the vertex from $V(A_2)\cap V(P_3)$ that minimizes the distance $d(a_2, x)$, and that $y$ is the vertex from $V(B_2) \cap V(P_3)$ that minimizes the distance $d(b_2, y)$.

Assume first that $B_1$ does not intersect $B_2 - u_0$.
Note that by \Cref{rem: A1}, it implies that $y$ is in $V(P_2) \cap V(P_3) \setminus V(P_1)$.
We set $D_2 := P_2[y,b_2]$ and $D_3 := P_3[a_3,y]$.
By our minimality assumption, $y$ is must then be the only intersection of $D_2$ with $P_1\cup P_3$.
By \Cref{lem:alternatively intersect}, no vertex on $D_3$ can belong to $B_1$.
Hence, we have that $P_1$ does not intersect $D_2$ nor $D_3$. By \Cref{lem:not intersect both}, with $y$ playing the role of $u_0$, $P_1$ playing the role of $P_3$ and $D_2$ playing the role of $A_1$, we obtain that $\pw(G) \leq 3$.  

We now assume that $B_1$ intersects $B_2 - u_0$. In particular, $B_1$ does not intersect $A_2 - u_0$ by \Cref{clm: ADN}.  By our choice of $x$ and \Cref{rem: A1}, observe that $C_2 := P_2[a_2, x]$ only intersects $P_1\cup P_3$ in $x$. We now set $A_3:=P_3[a_3, x]$. Note that by \Cref{clm: ADN} and our choice of $x$, $A_3$ only intersects $P_2$ in $x$. Moreover, note that as $A_3$ is disjoint from $A_1$, if $A_3$ intersects $B_1$, then $P_3$ bounces between $B_1$ and $A_2$, so by \Cref{lem:alternatively intersect} we have $\pw(G)\leq 3$. We may thus assume that $x$ is the vertex in the intersection of $A_3$ and $P_1\cup P_2$. Note that we are then in the situation of \Cref{lem:not intersect both}, with $x$ playing the role of $u_0$, $P_1$ playing the role of $P_3$ and $C_2, A_3$ playing the respective roles of $A_1, A_2$. We thus conclude that $\pw(G)\leq 3$. 
\end{proof}

\begin{proof}[Proof of \Cref{thm: 3-paths}]
As explained above, observe that, up to replacing $P_2$ by $P_2^{-1}$, one of the four cases between \ref{it: 1}, \ref{it: 2}, \ref{it: 3} and \ref{it: 4} should occur. We thus conclude that $\pw(G)\leq 3$ after applying according accordingly \Cref{lem:alternatively intersect}, \Cref{lem:not intersect both}, \Cref{lem:intersection semi alternatively} and \Cref{lem:intersect twice then once}.
\end{proof}

\section{Coverings by Isometric Subtrees}
\label{sec: trees}

\paragraph*{Treewidth.}A \emph{tree decomposition} of a graph $G$ is a pair $(T,\mathcal V)$ where $T$ is a tree and $\mathcal V=(V_t)_{t\in V(T)}$ is a family of subsets $V_t$ of $V(G)$ such that:
\begin{itemize}
 \item $V(G)=\bigcup_{t\in V(T)}V_t$;
 \item for every nodes $t,t',t''$ such that $t'$ is on the unique path of $T$ from $t$ to $t''$, $V_t\cap V_{t''}\subseteq V_{t'}$; 
 \item every edge $e\in E(G)$ is contained in an induced subgraph $G[V_t]$ for some $t\in V(T)$. 
\end{itemize}

The subsets $V_t$ are called the \emph{bags} of the tree decomposition $(T,\mathcal V)$, and the tree $T$ is called the \emph{decomposition tree}. The \emph{width} of a tree decomposition is the maximum size of a bag minus $1$. The \emph{treewidth} of a graph $G$, denoted $\tw(G)$ is defined as the minimum possible width of a tree decomposition of $G$. Note that path decompositions of $G$ correspond exactly to the tree decompositions of $G$ whose decomposition tree is a path. In particular, we have in general $\tw(G)\leq \pw(G)$. 

Recall that a graph $G$ is \emph{$2$-connected} if it is connected of order at least $3$, and has no \emph{cutvertex}, that is, no vertex $v$ such that $G-v$ is not connected anymore. The \emph{blocks} of  $G$ are the inclusion-wise maximal sets of vertices $X\subseteq V(G)$ such that $G[X]$ is connected and has no cutvertex\footnote{Note that here, blocks are not necessarily $2$-connected, as we allow them to have order less than $3$.}. The following is a well-known and easy fact. 

\begin{proposition}[folklore]\label{prop:tw_blocks}
  The treewidth of a graph is equal to the maximum treewidth of its blocks (maximal $2$-connected components).
  Put differently, for every $v\in V(G)$, if $X_1, \ldots, X_k$ denote the connected components of $G-v$, then $\tw(G)= \max_{1\leq i\leq k} \tw(G[X_i\cup \sg{v}]).$
\end{proposition}

The following remark is immediate.

\begin{remark}
 \label{rem: blocks-2cover}
 Let $G$ be a graph edge-coverable by $k$ isometric trees, and $X$ be a block of $G$. Then $G[X]$ is also edge-coverable by $k$ isometric trees.
\end{remark}

We now prove our main result in this section, which we restate.
\mainTrees*

\begin{proof}
  Let $T_1,T_2$ denote two isometric subtrees of $G$ that cover all its edges. By \Cref{prop:tw_blocks} and \Cref{rem: blocks-2cover}, we may assume without loss of generality that $G$ is $2$-connected. In particular, $G$ has no vertex of degree $1$.
  Moreover, up to removing some edges incident to the leaves of $T_1$ and $T_2$, we may further assume that every edge of $T_1$ which is incident to a leaf of $T_1$ is not an edge of $T_2$.
  Let $u$ be a leaf of $T_1$. As $G$ has no vertex of degree $1$, $u$ must also be a vertex of $T_2$.
  
  \begin{claim}
  \label{clm: vertical}
   Every edge of $G$ is vertical with respect to $u$. 
  \end{claim}
  
\begin{proof}[Proof of the Claim]
    Suppose for a contradiction that there exists some edge $vw\in E(G)$ which is not vertical with respect to $u$. Then it implies that $v$ and $w$ belong to the same layer of the BFS-layering starting from $u$, that is, $d_G(u,v)=d_G(u,w)$. We let $i\in \sg{1,2}$ be such that $vw\in E(T_i)$. 
    As $T_i$ is an isometric subtree of $G$, we then also have $d_{T_i}(u,v)=d_{T_i}(u,w)$. We then obtain a contradiction as such an edge $vw$ would create a cycle in $T_i$. 
\end{proof}
  
  We now consider the subtree $T'_1$ of $T_1$ formed by taking the union of all paths $P$ between $u$ and a vertex of $T_1$ such that all internal vertices of $P$ are in $V(T_1)\setminus V(T_2)$.
  As $G$ is $2$-connected, all leaves of $T'_1$ are in $T_2$.
  Let $G_v$ be the subgraph of $G$ induced by the descendants of $v$, i.e., by vertices $x$ for which there exists some $ux$-path $P$ which contains $v$ and which is vertical with respect to $u$. Note that in particular, $v\in V(G_v)$.
  \begin{claim}\label{clm: cutvertex}
    Each leaf $v$ of $T'_1$ which is distinct from $u$ separates $G_v$ from the rest of $G$.
  \end{claim}
  
  \begin{proof}[Proof of the Claim]
      Let $v$ be a leaf of $T'_1$ distinct from $u$, and suppose for a contradiction, that $v$ does not separate $G_v$ from $G-G_v$. Then there exists some edge $e=xy\in E(G)$ with $x\in V(G_v)\setminus \sg{v}$ and $y\in V(G)\setminus V(G_v)$. By \Cref{clm: vertical}, $e$ is vertical with respect to $u$. As we assumed that $y\notin V(G_v)$, we must have $d_G(u,y)<d_G(u,x)$. Let $i\in \sg{1,2}$ be such that $xy\in E(T_i)$. We let $P$ denote the path in $T_i$ connecting $u$ to $x$, which, by \Cref{clm: vertical}, is vertical with respect to $u$ and thus contains $y$. As $y$ is not in $G_v$, observe that $P$ cannot contain $v$. See \Cref{fig: tree-cutvertex} for the case where $i=2$.  
    We let $w$ be the least common ancestor in $T_i$ of $v$ and $x$ (recall that by definition of $T'_1$, $v$ is in $V(T_1)\cap V(T_2)$). As $P$ avoids $v$, we have $v\neq w$.
    As $T_i$ is an isometric subgraph of $G$, we then have $d_G(v,x)=d_{T_i}(v,x)= d_{T_i}(v,w)+d_{T_i}(w,x)>d_{T_i}(w,x)=d_G(w,x)$. 
    Now, recall that $x\in V(G_v)$, hence there exists some $ux$-path in $G$ from $v$ to $x$ going through $v$, which is vertical with respect to $u$. In particular, as $w$ is an ancestor of $v$ in $T_i$, it then implies that $d_G(v, x)\leq d_G(w,x)$, giving a contradiction (see \Cref{fig: tree-cutvertex}).
\end{proof}
\begin{figure}[htb]
    \centering
    \includegraphics{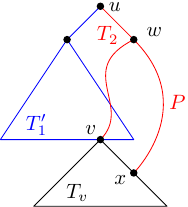}
    \caption{Proof of \Cref{clm: cutvertex} in the case where $xy\in E(T_2)$.}
    \label{fig: tree-cutvertex}
\end{figure}

  As $G$ is $2$-connected, \Cref{clm: cutvertex} implies that for each leaf $v$ of $T'_1$ distinct from $u$, we have $V(G_v)=\sg{v}$. In particular, as we initially chose $u$ to be a leaf of $T_1$, it then implies that $T_1=T'_1$, so the only vertices of $T_1$ that also belong to $T_2$ are its leaves. Observe that up to reproducing symmetrically the same reasoning, where the roles of $T_1$ and $T_2$ are exchanged, we may moreover assume that the only vertices of $T_2$ belonging to $T_1$ are its leaves. In other words, 
  we are left with the case where $T_1$ and $T_2$ have no common internal nodes, and intersect exactly at their leaves.
  
  In the remainder of the proof, we will say that a graph $H$ has a \emph{mirror-decomposition} if there exist trees $T, T'$ such that 
  \begin{itemize}
   \item $H=T\cup T'$;
   \item the trees $T$ and $T'$ have no common internal nodes;
   \item there exists a graph isomorphism $\iota$ from $T$ to $T'$ such that for every $v\in V(T)\cap V(T')$, we have $\iota(v)=v$. 
  \end{itemize}
  
  The proof of \Cref{thm: 2-trees} will be an immediate consequence of the following two Claims.
  
  \begin{claim}
  \label{clm: reconstruction}
    $G$ has a mirror-decomposition, with respect to the trees $T_1$ and $T_2$.
  \end{claim}
  
  \begin{proof}[Proof of the Claim]
      First, observe that as $T_1, T_2$ are both isometric subgraphs of $G$, for every two vertices $a,b$ in $V(T_1)\cap V(T_2)$, we have $d_G(a,b)=d_{T_1}(a,b)=d_{T_2}(a,b)$. In particular, the claim is then immediate if $T_1$ and $T_2$ have at most $2$ leaves. We now assume that $T_1, T_2$ have more than $2$ leaves.
    
    We recall that a basic property of trees is that for every three nodes $a,b,c$ in a tree $T$, there exists a unique \emph{median}, that is, a unique vertex $m(a,b,c)$ of $T$ that belongs to shortest paths between each pair of $a,b$ and $c$. 
    Note that in particular, if $a,b,c$ are pairwise distinct leaves of a tree $T$, then their median is an internal node of $T$.
    
    In what follows, for every $i\in \sg{1,2}$ and every three leaves of $T_i$ (and thus also of $T_{3-i}$), we let $m_i(a,b,c) \in V(T_i)$ denote their median in the tree $T_i$.
    We first claim that for every $x\in\sg{a,b,c}$, we have $d(m_1(a,b,c), x)=d(m_2(a,b,c), x)$.
    To see this, fix $a,b,c$ and for each $x\in \sg{a,b,c}$,
    set $d_{i}(x):=d_{T_i}(m_{i}(a,b,c),x)$. Note first that as $T_1, T_2$ are isometric subgraphs in $G$, we have for every two distinct vertices $x,y\in \sg{a,b,c}$ that $d_1(x)+d_1(y)=d_{T_1}(x,y)=d_{T_2}(x,y)=d_2(x)+d_2(y)$.
    In particular, we then have
    \begin{align*}
     2(d_{1}(a)+d_{1}(b)+d_{1}(c))  &= d_{T_1}(a,b)+d_{T_1}(b,c)+d_{T_1}(a,c) &&\\ \nonumber &=d_{T_2}(a,b)+d_{T_2}(b,c)+d_{T_2}(a,c) &&\\ \nonumber &=2(d_{2}(a)+d_{2}(b)+d_{2}(c)).
    \end{align*}
    By a previous observation, $d_{1}(b)+d_{1}(c)=d_{2}(b)+d_{2}(c)$, hence in particular it implies that $d_{1}(a)=d_{2}(a)$. The equalities $d_{1}(b)=d_{2}(b)$ and $d_{1}(c)=d_{2}(c)$ are obtained identically.

To conclude the proof of \Cref{clm: reconstruction}, we now show that for every tree $T$, the knowledge for each triple $(a,b,c)$ of distinct leaves of the tuples $$(d(m(a,b,c),a),d(m(a,b,c),b),d(m(a,b,c),c))$$ can be used to reconstruct $T$ uniquely. In other words, for a fixed set $L$ of size at least $3$ and a fixed mapping $\varphi: L^3\to \mathbb N^3$, there exists at most one tree $T$ (up to relabelling the internal nodes) such that the leaves of $T$ are exactly the elements of $L$, and such that for every triple $(a,b,c)$ of distinct leaves, we have 
    $$\varphi(a,b,c)=(d(m(a,b,c),a),d(m(a,b,c),b),d(m(a,b,c),c)).$$
    When such an equality holds for all triple of distinct leaves, we say that $T$ \emph{realizes} $\varphi$.
    Note that by previous paragraphs proving such a statement will immediately conclude our proof of \Cref{clm: reconstruction}. 
    We show this statement by induction on the size of $L$. If $|L|=3$, then note that the only possible tree realizing some given triple $\varphi(a,b,c)=(d_a, d_b, d_c)$ is a subdivided star with three branches of size $d_a, d_b$ and $d_c$ respectively connecting $a,b$ and $c$ to the center. Assume now that $|L|\geq 4$ and that the induction hypothesis holds. Suppose that we are given some tree $T$ with set of leaves $L$ realizing some fixed mapping $\varphi: L^3\to \mathbb N^3$. Let $a$ be a leaf of $T$, and let $P$ be the shortest possible path in $T$ connecting $a$ to a vertex $x$ of degree at least $3$ (such a vertex always exists as soon as $T$ has at least $3$ leaves). Then $T-P$ is a tree whose set of leaves is exactly $L\setminus\sg{a}$. In particular, applying induction hypothesis on $L\setminus \sg{a}$ implies that $T-P$ is the unique possible tree realizing the restriction of $\varphi$ on $(L\setminus \sg{a})^3$. Moreover, as $x$ has degree at least $3$ in $T$, there exist two distinct leaves $b,c$ of $T-P$ such that $x=m(a,b,c)$ in $T$. Write $\varphi(a,b,c)=(d_a, d_b, d_c)$ and observe now that $x$ is necessarily the only vertex of $T-P$ that lies on the $bc$-path of $T-P$ such that $d(x,b)=d_b$ and $d(x,c)=d_c$. In particular, it implies that the only possibility for a tree to realize $\varphi$ is to be obtained from $T-P$ after attaching some path $P'$ of length $d_a$ connecting $a$ to $x$. This concludes the proof. 
  \end{proof}
  
  \begin{claim}
   \label{clm: mirror-tw2}
   Every graph admitting a mirror-decomposition has treewidth at most $2$. 
  \end{claim}
  
  \begin{proof}[Proof of the Claim]
     We let $H$ be a graph with a mirror-decomposition and $T,T', \iota$ be as in the definition of mirror-decomposition.  We prove by induction on $|V(T)|=|V(T')|\geq 1$ that there exists a tree decomposition of $H$ of width at most $2$, and such for each $v\in V(T)$, there exists a bag that contains $\sg{v,\iota(v)}$.   
   If $|V(T)|\leq 2$, then $|V(H)|\leq 3$ (unless $V(T)$ and $V(T')$ are disjoint, in which case we also conclude immediately), and one can conclude by considering the tree decomposition with a single bag equal to $V(H)$.
   Assume now that $|V(T)|=|V(T')|\geq 3$. In particular, there exists some node $u\in V(T)$ such that $T-u$ is not connected. Then, $\sg{u,\iota(u)}$ forms a separator of size $2$ in $H$ (see \Cref{fig:mirror-trees}). 
   
   Now, observe that for each component $C$ of $T-u$, the graph $H_C:=H[C\cup \iota(C)\cup \sg{u,\iota(u)}]$ admits a mirror-decomposition with respect to the trees $T[C\cup \sg{u}]$ and $T'[\iota(C\cup \sg{u})]$. By choice of $u$, the tree $T[C\cup \sg{u}]$ has strictly fewer vertices than $T$, thus the induction hypothesis implies that $H_C$ admits a tree decomposition $(T_C,\mathcal V_C)$ such that for every $x\in C\cup \sg{u}$, there is a bag containing $\sg{x, \iota(x)}$. We let $t_C\in V(T_C)$ be a node of $T_C$ whose bag contains $\sg{u, \iota(u)}$. We now conclude the proof of the claim by considering the tree decomposition $(\widetilde T,\widetilde{\mathcal V})$ of $H$, where 
   $\widetilde T$ is obtained after taking the disjoint union of the trees $T_C$, and adding a vertex $z$ connected to all nodes $t_C$, with associated bag $\sg{u, \iota(u)}$. 
  \end{proof}
  \begin{figure}[htb]
    \centering
    \includegraphics[scale=0.8]{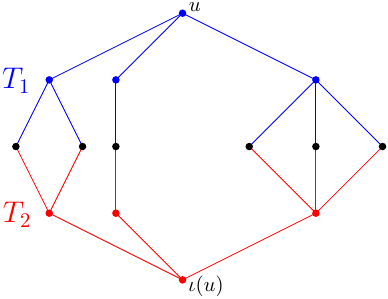}
    \caption{Example of mirrored trees. Notice that $\{u,\iota(u)\}$ indeed forms a separator of the graph.}
    \label{fig:mirror-trees}
  \end{figure}
  
\end{proof}

\section{A linear lower bound}
\label{sec: lb}

In this section, we prove a general lower bound on the pathwidth of graphs edge-coverable by $k$ shortest paths. 

\mainlb*

The remainder of the section is dedicated to the proof of \Cref{prop:low_bound}. We will construct inductively a sequence $(G_k)_{k\geq 1}$ of graphs satisfying the outcome of \Cref{prop:low_bound}.

\paragraph*{Construction of the sequence $(G_k)_{k\geq 1}$.}
We first consider the sequence $(\ell_k)_{k\geq 1}$ of integers defined inductively by setting $\ell_1:=3, \ell_2:=7$ and $\ell_{k+1}:=2(\ell_k+3)$ for each $k\geq 2$.
We now construct the sequence $(G_k)_{k\geq 1}$ inductively, maintaining at each step the property that $G_k$ admits a $k$-layering with $\ell_k$ layers, in which both the left and right extremities contain a single vertex, respectively denoted $u_k$ and $v_k$.

We first define separately $G_1$ and $G_2$: $G_1$ is the path of length $3$, with extremities $u_1,v_1$, and $G_2$ is the cycle of size $14$, seen as the union of two internally disjoint $u_2v_2$-paths of length $7$ (see \Cref{fig: Gk}). We now let $k\geq 2$ and construct inductively $G_{k+1}$. First, we let $H, H'$ denote two disjoint copies of $G_k$, such that $a_k, b_k$ (resp.\ $a'_k, b'_k$) denote the copies of the vertices $u_k,v_k$ in $H$ (resp.\ in $H'$). To obtain $G_{k+1}$, we moreover add an edge between $b_k$ and $a_k'$, two vertices $u_{k+1}$ and $v_{k+1}$, two paths $A_k, B_k$ of length $2$ respectively between $u_{k+1}$ and $a_k$, and between $v_{k+1}$ and $b_k'$, and finally, two paths $C_k, D_k$ of length $l_k+3$ respectively between $u_{k+1}$ and $a_k'$, and between $v_{k+1}$ and $b_k$, whose internal vertices are distinct from all the ones that were present so far (see \Cref{fig: Gk_general}).

An easy induction on $k$ shows that $G_k$ admits a $k$-layering $L_k$ containing $\ell_k$ layers for each $k\geq 1$, whose extremities are $\sg{u_k}$ and $\sg{v_k}$.

\begin{figure}
  \centering
  \includegraphics[scale=0.7]{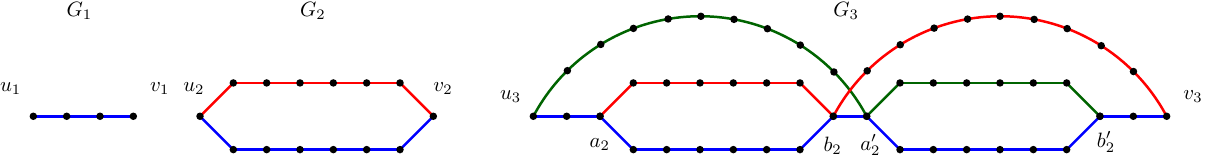}
  \caption{The graphs $G_1,G_2$ and $G_3$. The colored paths form a covering by $k$ shortest paths for each $k\in [3]$.}
  \label{fig: Gk}
\end{figure}

\begin{figure}
  \centering
  \includegraphics[scale=1]{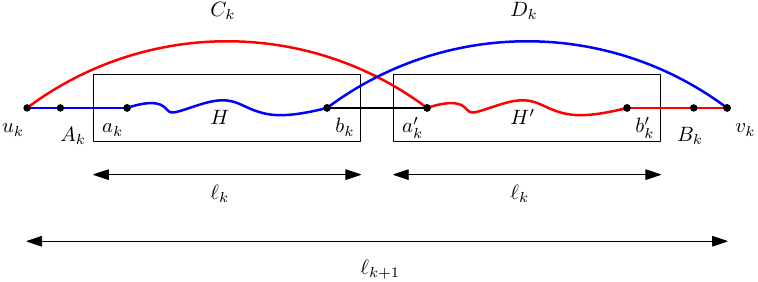}
  \caption{Recursive construction of $G_{k+1}$. The paths $Q_k, Q_{k+1}$ are respectively represented in red and blue.}
  \label{fig: Gk_general}
\end{figure}
 
\paragraph*{Proof of \ref{enum:coverable} in \Cref{prop:low_bound}.}
We show by induction that $G_k$ is edge-coverable by $k$ shortest paths.
We will moreover show that one can choose all of these paths to be $u_kv_k$-paths. In particular, it implies that they all contain exactly one vertex in each layer of $L_k$. The result is immediate for $k\leq 2$, so we now let $k\geq 2$ and assume that it holds for $G_k$. We let $P_1,\ldots, P_k$ and $P'_1,\ldots,P'_k$ denote respectively paths that edge-cover $H$ and $H'$, such that for each $i\in [k]$, $P_i$ (resp. $P'_i$) is a shortest $a_kb_k$-path in $H$ (resp. a shortest $a'_kb'_k$-path in $H'$). For each $i\in [k-1]$, we set $Q_i:= A_i\cdot P_i\cdot b_ka_k \cdot P'_i\cdot B_i$. Moreover, we set $Q_{k}:=C_k\cdot P'_k\cdot B_k$ and $Q_{k+1}:=A_k\cdot P_k\cdot D_k$ (see \Cref{fig: Gk_general}). It is easy to check that $Q_1,\ldots, Q_k$ edge-cover $G_{k+1}$. Moreover, note that each $Q_i$ has at most one vertex in each layer of $L_{k+1}$, implying in particular that it is a shortest $u_{k+1}v_{k+1}$-path, which concludes the proof. 

\paragraph*{Proof of \ref{enum:pathwidth} in \Cref{prop:low_bound}.}
We now show that for each $k\geq 1$, we have $\pw(G_k)\geq k$, which will prove \ref{enum:pathwidth} in \Cref{prop:low_bound} and therefore complete its proof.

\begin{figure}
  \centering
  \includegraphics[scale=0.8]{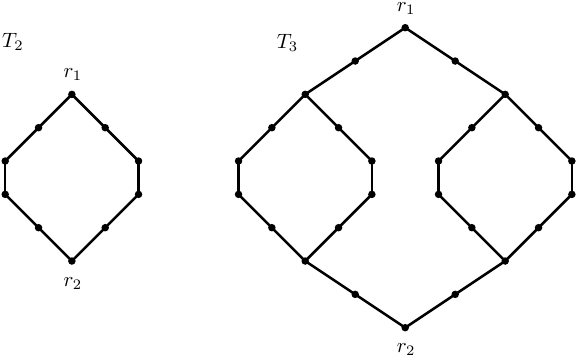}
  \caption{The graphs $T_2$ and $T_3$.}
  \label{fig: Tk}
\end{figure}

For each $k\geq 1$, we let $T_k$ denote the graph obtained as follows: take two disjoint copies $B_1$ and $B_2$ of a complete binary tree of height $k-1$, with respective roots $r_1,r_2$, subdivide all edges once, and add a matching connecting each leaf of $B_1$ to its copy  in $B_2$ (see \Cref{fig: Tk}). Note first that for each $k\geq 1$, $T_k$ is a minor of $G_k$. To see this, observe that an easy induction on $k$ further gives that there exists for each $k\geq 1$ a model of $T_k$ in which the branch sets of the respective roots of the copies $B_1, B_2$ of the binary tree are exactly $\sg{u_k}$ and $\sg{v_k}$ (note also that the edge $b_ka'_k$ can be ignored for this part, i.e.\ $T_k$ is also a minor of $G_k-a_kb_k$).
In particular, as $T_k$ is a minor of $G_k$, we have for each $k\geq 1$ that $\pw(G_k)\geq \pw(T_k)$. It thus only remains to show that $T_k$ has pathwidth at least $k$.
\begin{lemma}\label{lem: lowerbound}
  For any integer $k$, the graph $T_k$ has pathwidth at least $k$.
\end{lemma}
To prove \Cref{lem: lowerbound}, we use the game characterization of pathwidth, in terms of 
the \emph{node searching number}~\cite{KIROUSIS1985181}. We will reuse in what follows some of the vocabulary from \cite{dendris1997fugitive}.
 
\paragraph*{Node searching number.}
The definition of the node searching number of a graph $G$ requires to define the \emph{node searching game}, which is a pursuit-evasion game  on $G$. Intuitively, in this game, $k$ cops try to catch an invisible robber moving on the vertices of $G$, where, by invisible, we mean that before taking a decision, the robber already knows the whole sequence of moves that the cops will follow in the future, and can adapt its own strategy with respect to this knowledge. More precisely, a \emph{search with $k$ cops} on $G$ is a finite sequence $C_1, \dots, C_t$ of subsets of vertices of $G$ of size at most $k$ such that, for $i=2, \dots,t$, the symmetric difference between $C_i$ and $C_{i-1}$ has size at most $1$. Each set $C_i$ represents vertices
occupied by $k$ or fewer cops, and at each step, a cop is either added or removed (or nothing changed). 
Given such a search on $G$, we define inductively its corresponding \emph{free locations} $F_1,\ldots, F_t$ by setting $F_1:=V(G)\setminus C_1$ and for each $i\in \sg{2,\ldots,t}$, 
by letting $F_i$ be the set of all vertices $v$ from $V(G)\setminus C_i$ for which there exists some path from some vertex $u\in F_{i-1}$ to $v$, with all vertices except possibly $u$ in $V(G)\setminus C_i$. Informally, $F_i$ represents the set of all vertices that could be reached by the robber at step $i$, while evading the cops so far. 
We say that a search $C_1,\ldots,C_t$ is \emph{evasive} if the robber can escape the cops with respect to this search, i.e.\ if $C_t\neq \emptyset$. 
The \emph{node searching number} of $G$, denoted $\ns(G)$ is the minimum possible integer $k\geq 0$ for which $G$ admits a non-evasive search with $k$ cops. Intuitively, $\ns(G)>k$ means that when playing with $k$ cops, whatever strategy the cops choose, the robber will be able to find accordingly a strategy to escape them. A classic result of Kirousis and Papadimitriou \cite{KIROUSIS1985181} states that for every graph $G$, we have $\ns(G)=\pw(G)+1$. In particular, in order to conclude our proof of \Cref{prop:low_bound}, it will be enough to show that for each $k\geq 1$, $\ns(T_k)\geq k+1$.

In order to show this, we need to slightly extend the definition of the game by allowing some winning positions for the robber. For a set $W\subseteq V(G)$ of vertices, which we call a \emph{winning state}, a search $C_1,\ldots,C_t$ on $G$ is \emph{$W$-winning} if there exists some $i\in [t]$ for which $F_i\cap W\neq \emptyset$. The \emph{$W$-node searching number} of $G$, denoted $\ns_W(G)$ is the minimum possible integer $k\geq 0$ for which $G$ admits a search which is neither $W$-winning, nor non-evasive. More informally, $\ns_W(G)>k$ means that when playing with $k$ cops, whatever strategy the cops choose, the robber will be able to find accordingly a strategy allowing him either to enter $W$ at some step, or to escape them forever. Combined with all previous observations, next lemma will immediately conclude our proof of \Cref{prop:low_bound}. In what follows, for each $k\geq 1$, we let $T'_k$ denote the graph obtained from $T_k$ after adding two new vertices $s_1, s_2$ of degree $1$ respectively attached to the roots $r_1,r_2$ of the binary trees used in the definition of $T_k$. 

\begin{lemma}\label{lem: lowerbound_induction}
  We have:
  \begin{enumerate}
    \item\label{it: ns1} $\ns(T_k)\geq k+1$ for every $k\geq 1$, and
    \item\label{it: ns2} $\ns_{\{s_1,s_2\}}(T_k') \geq k+2$ for every $k\geq 2$.  
  \end{enumerate}
\end{lemma}

\begin{proof}

  The proof goes by induction on $k\geq 1$. The base case $k=1$ is immediate, as $T_1$ is then an edge with extremities $r_1,r_2$, in which case it is easy to see that one cop can never catch the robber \footnote{Recall that in order to move, the cop must first disappear from the graph. This follows from the condition $|C_{i-1}\Delta C_i|\leq 1$ in the definition of a search.}, implying \Cref{it: ns1}. 
  $T_2$ is a cycle of size $10$ on which $r_1$ and $r_2$ are two antipodal vertices. In this case, it is not hard to check that a robber will always be able to escape to two cops in the node searching game showing \Cref{it: ns1} when $k=2$. 
  We now show that \Cref{it: ns2} also holds when $k=2$.  
  For this, we describe a winning strategy for the robber in the $\sg{s_1,s_2}$-node searching game with $3$ cops. Let $P, Q$ denote the two paths of length $3$ obtained when removing $r_1, r_2$ from $T_2$ and consider a search 
  $C_1,\ldots,C_t$ of $T'_2$. We also set $S_1:=\sg{r_1, s_1}$ and $S_2:=\sg{r_2, s_2}$. 
  First, note that if each set $C_i$ intersects both $S_1$ and $S_2$, then the robber has a winning strategy to evade the cops by choosing any of the two paths $P,Q$ and staying in it the whole time (such a path will be occupied by at most one cop at each step, allowing the robber to win). Assume now that this is not the case, and let $i\in [t]$ be the first step of time for which one of the two sets $S_1, S_2$, say $S_1$ does not intersect $C_i$. If $i=1$, then the robber immediately wins by choosing $s_1$ as its initial position, so we assume that $i\geq 2$. 
  Then for each $1\leq j<i$, both paths $P$ and $Q$ are occupied by at most one cop, and moreover one of them, say $P$ is occupied by no cop at step $i-1$ (if $i\geq 2$). In particular, as $|C_{i-1}\Delta C_i|\leq 1$ and as $S_1$ contains a vertex from $C_{i-1}\setminus C_i$, $P$ is also not occupied by a cop at step $i$. A winning strategy for the robber then consists in staying in $P$ during the first $i-1$ steps and escaping so far the only cop that enters $P$, and then reaching $s_1$ at step $i$. 
  We assume from now on that $k\geq 3$, and that \Cref{it: ns1,it: ns2} holds in $T_{k-1}$. 
  
  In what follows, we let $H_1, H_2$ denote the two disjoint copies of $T_{k-1}$ in $T_k$, and $u_i^j$ ($i,j\in [2]$) denote the copies of $r_1,r_2$ in $H_1, H_2$, as depicted in \Cref{fig: Proof_ns}. We moreover let 
  $H_1'$ and $H_2'$ denote the two disjoint copies of $T_{k-1}'$ in $T_k$, such that for each $i\in [2]$, $H'_i$ contains $H_i$, and for each $i,j\in [2]$, we let $v_i^j$ denote the copies of the vertices $s_1,s_2$ in $H_1', H_2'$, as depicted in \Cref{fig: Proof_ns}. We moreover set $W_j:=\sg{v_1^j, v_2^j}$ for each $j\in \sg{1,2}$.

   \begin{figure}
  \centering
  \includegraphics[scale=1]{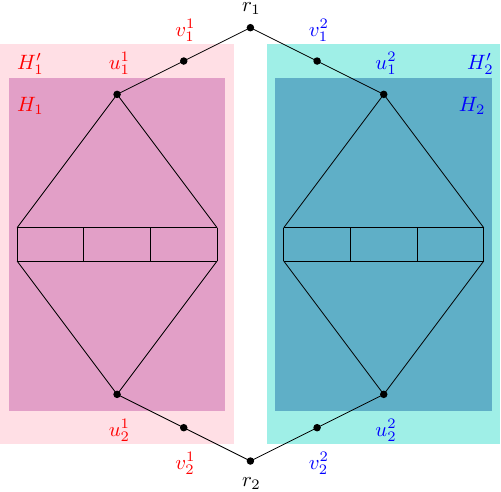} 
  \caption{Schematic description of the decomposition of $T_k$ in the proof of \Cref{lem: lowerbound_induction}. The triangles represent copies of the complete binary trees of heigth $k-2$ involved in the definition of $T_{k-1}$.}
  \label{fig: Proof_ns}
\end{figure}

We start proving \Cref{it: ns1} and consider a search $C_1, \ldots,C_t$ on $G$ with $k$ cops. We will show that this search is evasive. Intuitively, the strategy simply consists for the robber in playing the $W_j$-node search game in some subgraph $H'_j$ ($j\in [2]$). Applying \Cref{it: ns2} from induction hypothesis on $H'_j$, the robber will be able either to escape the cops until the end of the search, or to jump into $H'_{3-j}$ by travelling through some vertex of the set $W_j$, and then start playing iteratively the $W_{3-j}$-node search game in $H'_{3-j}$. This is slightly more technical to prove formally, as one must ensure that when jumping from $H'_j$ to $H'_{3-j}$, the robber can moreover reach some initial configuration of the winning strategy he will use later in $H'_{3-j}$. 

More formally, for each $i\in [t]$, we construct the \emph{projection $\pi_1(C_i)\in \binom{V(H'_1)}{\leq k}$} of $C_i$ on $H'_1$ as follows. First, for each $u\in C_i\cap V(H'_1)$, we add $u$ to $\pi_1(C_i)$. Then, if $C_i$ contains any of the vertices $r_1, v_1^2, u_1^2$ (these correspond to the vertices on the shortest $v_1^1$-$u_1^2$ path), we add $v_1^1$ to $\pi_1(C_i)$. Symmetrically, if $C_i$ contains any of the vertices $r_2, v_2^2, u_2^2$, we add $v_2^1$ to $\pi_1(C_i)$. If necessary, for each vertex $u$ in $C_i\cap H_2\setminus\sg{u_1^2, u_2^2}$, we add at least one of the two vertices $v_1^1, v_2^1$, by respecting the following rule that if $H_2$ contains at least two vertices from $C_i$, then both $v_1^1$ and $v_2^1$ are added to $\pi_1(C_i)$. We claim that a simple induction on $i\in [t]$ implies that we can construct such sets $\pi_1(C_i)$ with the additional property that $\pi_1(C_1),\ldots,\pi_1(C_t)$ is a search on $H'_1$ (with at most $k$ cops). We define symmetrically a projected search $\pi_2(C_1),\ldots,\pi_2(C_t)$ on $H'_2$.

We now describe a winning strategy for the robber to show that $C_1,\ldots,C_t$ is evasive. We will ensure that at each step, the robber occupies a vertex of $H_1\cup H_2$. A \emph{safe state} is a configuration of the game where at some step $i\in [t]$, the robber occupies some vertex in some subgraph $H_j$ ($j\in [2]$) such that $|C_i\cap V(H_j)|\leq 1$. 
We will first show that we may assume that the robber can reach some safe state. 
We assume without loss of generality that $|V(H_1)\cap C_1|\leq k-1$. By induction hypothesis, \Cref{it: ns2} holds in $H'_1$, implying that the robber has a strategy to win the $W_1$-node searching game in $H'_1$ with respect to the projected search $\pi_1(C_1),\ldots, \pi_1(C_t)$. It then implies that one of the following holds:
\begin{itemize}
 \item either the robber can escape the cops with respect to $\pi_1(C_1),\ldots,\pi_1(C_t)$, in which case he also has a strategy to escape the cops in $T_k$ with respect to the initial search $C_1,\ldots,C_t$, and staying in $H'_1$ at every step,
 \item or there exists some step $i\in [t]$ such that the robber can stay in $H'_1$ and escape to the cops until step $i$, and then reach one of the two vertices of $W_1$. Without loss of generality, assume that this vertex is $v_1^1$. The definition of $\pi_1$ then implies in particular that $|C_i\cap V(H_2)|\leq 1$, and moreover that none of the vertices $v_1^1, r_1, v_1^2, u_1^2$ belongs to $C_i$. This then implies that the robber can safely travel to the vertex $u_1^2$ at step $i$, and reach a safe state. 
\end{itemize}
We now prove that if at some step $i\in [t]$, the robber could reach some safe state in which he occupies a vertex $x$ of $H_j$, then it can either escape the cops until the end of the search by staying in $H_j$, or reach another safe state at some step $i'>i$. Observe that applying iteratively this claim immediately concludes the proof of \Cref{it: ns1}. To prove the claim, the arguments are almost the same as above: assume that at step $i$, the robber occupies a vertex of $H_1$ and that $|V(H_1)\cap C_i|\leq 1$. If $i=t$, then we immediately conclude. Otherwise, consider the search $\pi_1(C_{i+1})\cup \sg{v_1^1, v_2^1}, \pi_1(C_{i+1}),\ldots, \pi_1(C_t)$ of $H'_1$ (note that this indeed defines a search as $\pi_1(C_{i+1})$ contains at least one of the vertices $v_1^1, v_2^1$). By induction hypothesis, the robber has a winning strategy for this search when playing the $W_1$-node searching game in $H'_1$, for which its initial position is some vertex $y\in V(H_1)$. We now observe that as $H_1$ is $2$-connected, and as $|V(H_1)\cap C_i|\leq 1$, the robber can, at the end of step $i$, move from $x$ to $y$ without being caught by the cops. He can then reproduce the winning strategy in $H'_1$ for the $W_1$-node searching game, and according to the outcome, the same arguments used in the previous case distinction imply that he can either escape the cops until the last step, or reach another safe state for some $i'>i$.

We now show \Cref{it: ns2}, using \Cref{it: ns1} from the induction hypothesis. The proof is basically the same as for the case $k=2$. Let $C_1,\ldots,C_t$ be a search with $k+1$ cops on $T'_k$. We let $W:=\sg{s_1, s_2}$, $S_1:=\sg{v_1^1, r_1, s_1, v_1^2}$ and $S_2:=\sg{v_2^1, r_2, s_2, v_2^2}$, so that $H_1, H_2$ are the two connected component of $G-(S_1\cup S_2)$. If for each $i\in [t]$, $|C_i\cap V(H_1)|\leq k-1$, the robber can escape to the cops by reproducing some winning strategy given by \Cref{it: ns1} in $H_1$, implying that $C_1,\ldots,C_t$ is evasive. The same applies symmetrically if $|C_i\cap V(H_2)|\leq k-1$ for each $i\in [t]$. Otherwise, we can consider the minimum $i\in [t]$ such that $|C_i\cap V(H_1)|=k$, and we may moreover assume that $|C_{i'}\cap V(H_2)|\leq k-1$ for each $i'\leq i$. In particular, $C_i$ must be disjoint to at least one of the sets $S_1, S_2$. If $i=1$, then the robber immediately wins the $W$-node search game by choosing $s_1$ as its initial position. Otherwise, by induction hypothesis, \Cref{it: ns1} holds in $H_2$ and implies that the robber has a winning strategy to escape the cops until step $i$, while staying in $H_2$. In particular, as $|C_i\cap V(H_1)|=k$, we then have $|C_i\cap V(H_2)|\leq 1$. As $H_2$ is $2$-connected, it then implies that at the start of step $i$, the robber can safely reach one of the sets $S_1, S_2$ which is disjoint from $C_i$, and thus also $W$. This concludes the proof of \Cref{it: ns2}.

\end{proof}

\section{Conclusion}
\label{sec: ccl}

As mentioned in the introduction, the best lower bound we know on the possible pathwidth of graphs edge-coverable by $k$ shortest paths is $k$. Together with \Cref{thm: 3-paths}, it suggests that the upper bound from \Cref{thm: main} could be potentially improved to a linear bound.

\begin{question}
 \label{question: linear}
 Let $G$ be a graph edge-coverable by $k$ shortest paths. Is it true that $\pw(G)=O(k)$? And that $\pw(G)\leq k$?
\end{question}

Another observation going in the direction of \Cref{question: linear} is that graphs edge/vertex-coverable by $k$ shortest paths have \emph{cop number} at most $k$~\cite{AF84}. In particular, it is well-known (and not hard to prove) that the cop number of a graph is upper bounded by its treewidth, and thus also by its pathwidth. 

The second question which is left open by our work (and asked in ~\cite{DFPT24}) concerns the existence of a polynomial upper-bound on the pathwidth of graphs which are vertex-coverable by $k$ shortest paths. At the moment, it does not seem for us that the ideas from \Cref{sec: poly} generalize to this case.

Another natural question consists in asking if \Cref{thm: 2-trees} generalizes in some way to larger values of $k$: 

\begin{question}
 \label{question: trees} 
 Does there exist a function $f: \mathbb N\to \mathbb N$ such that for every 
 $k\geq 3$ and every graph $G$ edge-coverable by $k$ isometric subtrees, we have $\tw(G)\leq f(k)$? If yes, can $f$ be polynomial?
\end{question}

However, Bastide, Duron, Hodor, Liu and Nie~\cite{BDHLN} found constructions of graphs edge-coverable by $4$ isometric subtrees that contain arbitrarily large subdivided walls as a subgraph, implying a negative answer to Question \ref{question: trees}. A question left open by their work concerns the existence of graphs edge-coverable by $3$ isometric subtrees and with arbitrarily large treewidth. 
It might also be interesting to find some restricted hypothesis under which \Cref{question: trees} could hold.

\section*{Acknowledgements}
We would like to thank the anonymous reviewers for their instructive feedbacks about the conference version of this paper. In particular one of them pointed to us that the lower bound from \Cref{prop:low_bound} did not follows from \cite{DFPT24} as we had claimed in a previous version of the paper.
We thank Marcin Bria\'nski for the example from \Cref{sec: intro} (see \Cref{fig:vertex-coverable-trees}) of graphs vertex-coverable by two isometric trees with large treewidth. We also thank Florent Foucaud for mentioning the potential algorithmic application of \Cref{thm: main} to \textsc{Isometric Path Edge-Cover}.

\bibliographystyle{alpha}
\bibliography{biblio}

\end{document}